\documentclass[reqno,10pt]{amsart}
\usepackage{amsmath,mathrsfs}
\usepackage{amssymb, amsmath}
\usepackage{graphicx}
\usepackage{pdfsync}
\usepackage{mathtools}
\usepackage{color}
\usepackage[marginpar=2cm]{geometry}
\usepackage[colorlinks,
linkcolor=red,
anchorcolor=blue,
citecolor=green
]{hyperref}
\usepackage{hyperref}
\usepackage{colonequals,todonotes}
\def\inte#1{
	\displaystyle\mathop{#1\kern0pt}^\circ }

\let\pa=\partial
\let\al=\alpha
\let\b=\beta

\let\d=\delta

\let\f=\frac
\let\vf=\varphi

\def\pa{\partial}
\def\grad{\nabla}

\def\virgp{\raise 2pt\hbox{,}}
\def\cdotpv{\raise 2pt\hbox{;}}

\def\C{\mathop{\mathbb C\kern 0pt}\nolimits}
\def\DD{\mathop{\mathbb D\kern 0pt}\nolimits}
\def\EE{\mathop{{\mathbb E \kern 0pt}}\nolimits}
\def\K{\mathop{\mathbb K\kern 0pt}\nolimits}
\def\N{\mathop{\mathbb N\kern 0pt}\nolimits}
\def\Q{\mathop{\mathbb Q\kern 0pt}\nolimits}
\def\R{\mathop{\mathbb R\kern 0pt}\nolimits}
\def\SS{\mathop{\mathbb S\kern 0pt}\nolimits}
\def\ZZ{\mathop{\mathbb Z\kern 0pt}\nolimits}
\def\TT{\mathop{\mathbb T\kern 0pt}\nolimits}
\def\P{\mathop{\mathbb P\kern 0pt}\nolimits}

\def\th{\theta}

\newcommand{\beq}{\begin{equation}}
	\newcommand{\eeq}{\end{equation}}
\newcommand{\ben}{\begin{eqnarray}}
	\newcommand{\een}{\end{eqnarray}}
\newcommand{\beno}{\begin{eqnarray*}}
	\newcommand{\eeno}{\end{eqnarray*}}
\makeatletter

\newcommand{\Rmnum}[1]{\expandafter\@slowromancap\romannumeral #1@}
\makeatother
\newtheorem{thm}{Theorem}[section]
\newtheorem{lem}{Lemma}[section]

\theoremstyle{definition}

\newtheorem{rmk}{Remark}[section]

\numberwithin{equation}{section}
\title[Optimal convergence estimates for Boltzmann equation]{Optimal convergence estimate of the limit from inverse power potential to hard sphere Boltzmann equation}

\author[Z.-N. Hu]{Zheng-nan Hu}
	\address[Z.-N. Hu]{School of Mathematics, Sun Yat-Sen University, Guangzhou, 510275, P. R.  China.}
	\email{huzhn3@mail2.sysu.edu.cn}
\author[J.-W. Jang]{Jin woo Jang}
\address[J.-W. Jang]{Department of Mathematics, POSTECH (Pohang University of Science and Technology), Pohang
37673, South Korea.}
\email{jangjw@postech.ac.kr}
    \author[Z.-A. Yao]{Zheng-an Yao}
	\address[Z.-A. Yao]{School of Mathematics, Sun Yat-Sen University, Guangzhou, 510275, P. R.  China.} \email{mcsyao@mail.sysu.edu.cn}
    
    \author[Y.-L. Zhou]{Yu-long Zhou}
	\address[Y.-L. Zhou]{School of Mathematics, Sun Yat-Sen University, Guangzhou, 510275, P. R.  China.} \email{zhouyulong@mail.sysu.edu.cn}

\begin{document}
 
	\begin{abstract} 
    %The inverse power law potential $U(r)=\frac{1}{r^{1/s}}, s \in(0,1)$ generates the Boltzmann kernel $B^{s}= |v-v_*|^{1-4s} b_{s}(\theta)$ which has some angular singularity as the derivation angle $\theta \to 0$. Jang et al. \cite{Jang2023-df} established 
	%  the limit from the singular kernel $B^{s}$ goes to the hard-sphere kernel $\frac{1}{4}|v-v_*|$
	%  as $s \to 0$, as well as weak limit of solutions, but without convergence rate.
	%   In this work, we first derive that $|b_{s}(\theta)-\frac{1}{4}| \leq C s \theta^{-2-2s}$. With this key estimate, we obtain a
	%   convergence rate $O(s)$ for
	%   solutions to the Cauchy problem on the homogeneous Boltzmann equation with arbitrarily large initial data in suitable Sobolev space.

      The inverse power potential $U(r)=r^{-1/s}, 0<s<1$,  generates the Boltzmann kernel 
$B^{s}=|v-v_*|^{1-4s} b_s(\theta)$ with an angular singularity as $\theta\to 0$.  
Jang et al.~\cite{Jang2023-df} proved the limit $B^{s}\to \frac14|v-v_*|$ as $s\to 0$, as well as weak convergence of solutions based on this kernel convergence.  
In this work we establish the following sharp quantitative estimate:
\[
|b_s(\theta)-\tfrac14| \le C\, s\,\theta^{-2-2s}.
\]
In particular, this sharp estimate yields the \emph{optimal} $O(s)$ convergence rate for solutions of the homogeneous Boltzmann equation with large initial data in suitable Sobolev spaces; i.e., for any $t\in[0,T]$, we have 
$$f^s(t)=f^0(t)+O(s),$$ quantified by the $L^1_k$ norm for $k\ge 2.$

\end{abstract}

\maketitle
%\allowdisplaybreaks
\section{Introduction}
\subsection{The Boltzmann equation and notations.} In the present paper, we study the spatially homogeneous Boltzmann equations with the exact collision kernels derived from the inverse power potentials $U(r)=\frac{1}{r^{1/s}}, 0<s<1$ and the hard sphere model. The homogeneous Boltzmann equations read
\begin{equation}\label{inverse-pl-s}
\begin{cases}
\pa_t f^s=Q^s(f^s,f^s)\\
f^s(0,v)=f_{\textup{in}}(v),
\end{cases}  
\end{equation}
and
\begin{equation}\label{hard-sphere}
\begin{cases}
\pa_t f^0=Q^0(f^0,f^0)\\
f^0(0,v)=f_{\textup{in}}(v).
\end{cases}
\end{equation}
Here, $f^s(t,v),f^0(t,v)\ge 0$ are the density functions of collision particles which at time $t\ge 0$ moving with velocity $v\in\mathbb{R}^3$. The Boltzmann collision operator $Q^s$ and $Q^0$ are associated to the kernel $B^s(v-v_*,\sigma), B^0(v-v_*,\sigma)$ respectively, which act only on velocity variable $v$, defined by
$$Q^s(g,f)=\int_{\mathbb{R}^3\times \mathbb{S}^2}B^s(v-v_*,\sigma)(g_*'f'-g_*f)\mathrm{d}\sigma\mathrm{d}v_*,$$
and
$$Q^0(g,f)=\int_{\mathbb{R}^3\times \mathbb{S}^2}B^0(v-v_*,\sigma)(g_*'f'-g_*f)\mathrm{d}\sigma\mathrm{d}v_*.$$
Here we use the standard shorthand $f=f(v), g_*=g\left(v_*\right), f^{\prime}=f\left(v^{\prime}\right), g_*^{\prime}=g\left(v_*^{\prime}\right)$, where $v',v'_*$ are given by 
$$v^{\prime}=\frac{v+v_*}{2}+\frac{\left|v-v_*\right|}{2} \sigma, \quad v_*^{\prime}=\frac{v+v_*}{2}-\frac{\left|v-v_*\right|}{2} \sigma, \quad \sigma \in \mathbb{S}^{2} .$$

%$=\frac{v'-v'_*}{|v'-v'_*|}$

In the hard-sphere model, the collision kernel $B^{0}$ depends only on the relative velocity and is given by (the radius of particle is $\frac{1}{2}$)
$$B^{0}(v-v_*,\sigma)=\frac{1}{4}|v-v_*|.$$
For the inverse power law $U(r)=\frac{1}{r^{1/s}}, 0<s<1$, 
explicit derivation of the Boltzmann kernel $B^{s}$ was given in \cite{Jang2023-df}. Here we recall the result in \cite{Jang2023-df} and introduce adequate notations therein for our purpose. For any $0<s<1$, 
the Boltzmann kernel  $B^{s}$  reads
\ben\label{1.3}
B^{s}(v-v_*,\sigma)=|v-v_*|^{\gamma}b_{s}(\theta)= |v-v_*|^{1-4s} b_{s}(\theta), \quad v,v_* \in \R^3, \theta \in [0,\pi],
\een
where the 
 angular part is given by:
\ben\label{b_s(theta)}
b_s(\th)=\frac{2^{4s}}{2}\frac{1}{\sin\th}\beta_s(y_s(\varphi))\beta_s'(y_s(\varphi))y_s'(\varphi).
\een
Here $v-v_*$ is the relative velocity and $\theta$ is the deviation angle: $\cos\theta = \frac{v-v_*}{|v-v_*|} \cdot \sigma$. In \eqref{b_s(theta)}, the two variables $\varphi$ and $\th$ are related by $2\varphi+\th=\pi$.   
The formula \eqref{b_s(theta)} involves two $s$-dependent functions $\beta=\beta_s(y), y=y_s(\varphi)$. The function $\beta_{s}$ is given by
\ben \label{function-beta-s-def}
\beta = \beta_{s}(y) \colonequals \frac{y}{(1-y^2)^{s}}, \quad  y \in [0,1].
\een
By direct computation, 
\ben \label{function-beta-s-derivative-def}
\beta_{s}^{\prime}(y) = 2s\frac{1}{(1-y^2)^{1+s}} + (1-2s)\frac{1}{(1-y^2)^{s}}
.
\een
The function $y_s$ is the inverse function of  $\vf_s$ defined by 
\ben \label{varphi-given-by-integral}
\varphi  =\varphi_{s}(y)  \colonequals y\int_{0}^{1}\frac{1}{\sqrt{g(s,y,z)}}\mathrm{d}z,\quad y \in[0,1].
\een 
Here, for the sake of simplicity, we define
\ben\label{def-of-g}
g(s,y,z)\colonequals 1-z^{1/s}-y^{2}(z^2-z^{1/s}), \quad y,z \in[0,1], s \in (0,1).
\een

%The limit $b_s(\theta) \to \frac{1}{4}$ as $s \to 0$ for any $0<\theta \leq  \pi$ has been proven  \cite{Jang2023-df}(see Theorem 1(i)). 
As $s \to 0$, the inverse-power potential converges to the hard-core (hard-sphere) potential in the sense that 
$U_s(r)=r^{-1/s} \to U_{0} \colonequals \infty\,\mathbf{1}_{0<r<1}$.
Consequently, the convergence of the singular Boltzmann kernel $B^{s}$ to the hard-sphere kernel 
$\frac{1}{4}|v-v_*|$ as $s \to 0$ has long been expected in the kinetic theory literature.  
As already anticipated in \cite[Remark 1.0.1]{GST2013}, such a limit was suggested to exist in their discussion of the hard-sphere and short-range potential regimes, and it was rigorously established for the first time in the work of Jang--Kepka--Nota--Vel{\'a}zquez~\cite{Jang2023-df}. 
The authors of \cite[Theorem 1(i)]{Jang2023-df} proved that the following limit
\ben \label{convergence-of-b-s-to-constant}
\lim_{s \to 0 } b(s,\theta) = \f{1}{4},
\een
holds for any $0<\theta \leq  \pi$.
However, \eqref{convergence-of-b-s-to-constant} is stated without an explicit
quantitative convergence rate uniform in \(\theta\). 
 In this work, we provide an optimal convergence rate for this limit. Furthermore, the convergence result $f^s\to f^0$, $s\to0$ (in some suitable norms) will be proved.
\subsection{Main results}
Our first main result provides a sharp and fully explicit quantitative version of the formal limit $b_s(\theta)\to\frac14$ as $s\to0$. 
Because the kernel is expressed through several implicitly defined functions whose $s$–dependence becomes singular near $\varphi=\pi/2$, obtaining an optimal estimate requires controlling all such compositions uniformly in $s$. 
The following theorem establishes the optimal convergence rate $O(s)$.

\begin{thm}\label{kernel-convergence}  For any  $0<\theta \leq  \pi, 0<s<1$, it holds that
	\ben\label{convergence_rate_in_global}
	\left|b_s(\th)-\frac{1}{4}\right|\leq 50000s\th^{-2-2 s}.
	\een
\end{thm}

Our proof of Theorem \ref{kernel-convergence} is purely explicit and constructive so that we have the constant $50000$.
Comparison between Theorem \ref{kernel-convergence} and the results in  \cite{Jang2023-df} is given in the following remark.
\begin{rmk}
[Optimality]\label{optimality}
We briefly explain why the estimate in Theorem~\ref{kernel-convergence} is optimal
both in the parameter $s$ and in the angular singularity.
Recalling from \cite[Theorem~1(ii)]{Jang2023-df} that for any fixed $0<s<1$,
\begin{equation}\label{convergence-of-b-s-theta-to-0}
    \lim_{\theta \to 0} \theta^{2+2s} b(s,\theta)
    = C_s 
    = 2^{4s}s \, \bigl(\varphi_s'(1)\bigr)^{2s}
    = 2^{4s}s \left(\frac{1}{s} W_{1/s}\right)^{2s},
\end{equation}
where $W_n=\int_0^{\pi/2} \sin^{n} t \mathrm{d}t$ denotes the classical Wallis integral.
Using the well-known asymptotic behavior of $W_n$ as $n\to\infty$, one verifies that
\begin{equation}\label{C-s}
    \lim_{s\to0} \frac{C_s}{s} = 1.
\end{equation}
Combining \eqref{convergence-of-b-s-theta-to-0} and \eqref{C-s}, we see that for sufficiently
small $s$ and $\theta$, the kernel satisfies
\[
    b_s(\theta)\sim s\,\theta^{-2-2s}.
\]
Thus the estimate \eqref{convergence_rate_in_global} in Theorem \ref{kernel-convergence}
achieves the optimal order in both variables $s$ and $\theta$; no better decay in $s$
or in the power of $\theta$ can be expected in general.

We also note that the uniform upper bound
\begin{equation}\label{upper-bound-uniform-in-theta-s}
    \sup_{s\in(0,1/2]}\;\sup_{\theta\in(0,\pi]}\;
    \theta^{2+2s} b(s,\theta) < \infty
\end{equation}
was established in \cite[Theorem~1(iii)]{Jang2023-df}.  
Our new estimate \eqref{convergence_rate_in_global} implies \eqref{upper-bound-uniform-in-theta-s} directly.
\end{rmk}

The above convergence estimate is significant in several respects. First, it clarifies the relationship between two fundamental physical models—the hard-sphere model and the inverse power model—by providing a precise rate at which the latter approaches the former. Such a quantitative link is essential for understanding scattering phenomena, since both models arise naturally in the derivation of microscopic scattering laws, and the convergence rate offers new insight into how hard-sphere interactions effectively approximate long-range collisions in inverse power potentials. 

 Second, the estimate serves as a basic ingredient for the study of the Boltzmann equation. In particular, Theorem \ref{kernel-convergence} enables the derivation of a convergence rate for solutions $f^s\to f^0$ in suitable norms, thereby providing a rigorous proof of the hard-sphere approximation at the equation level. This is exactly the following Theorem \ref{thm1.2}.

Let us introduce adequate notations of function norms which we will use throughout the paper.
\beno 
|f|_{L^1_{q}}=\int_{\mathbb{R}^3}|f(v)|\langle v\rangle^q\mathrm{d}v,\quad |f|_{W^{1,1}_q}=\sum_{|\al|\le1}|\pa^{\al}f|_{L^1_q},
\\ H(f) = \int_{\mathbb{R}^3}f(v)\log f(v)\mathrm{d}v,\quad  |f|_{L\log L}=\int_{\mathbb{R}^3}|f\log f|\mathrm{d}v,
\eeno
where $\langle v\rangle = (1+|v|^2)^{1/2}$.

The optimal kernel estimate in Theorem~\ref{kernel-convergence} allows us to compare the full Boltzmann operators $Q_s$ and $Q_0$ in a quantitative manner. 
To measure convergence at the level of solutions, we introduce the scaled difference
\[
F^s(t,v)=\frac{f_s(t,v)-f_0(t,v)}{s},
\]
which isolates the first-order deviation of $f_s$ from the hard-sphere solution $f_0$. 
Our second main result below shows that $F^s$ remains uniformly bounded in weighted $L^1_k$ norms for any $k\ge2$, yielding the optimal rate
\[
f_s(t)=f_0(t)+O(s)
\quad\text{in }L^1_k,\qquad t\in[0,T].
\]
This establishes the optimal convergence of solutions in the hard-sphere limit.

\begin{thm}\label{thm1.2} Let $2\le k\in\mathbb{R}$, $0<\delta\in\mathbb{R}$, suppose $0 \leq f_{\textup{in}} \in  L^1_{k+3+\delta} \cap W^{1,1}_{k+2+\delta} \cap L\log L$. For $0<s<\frac{1}{8}$, let $F^s=\frac{f^s-f^0}{s}$ where $f^s$ and $f^0$ are solutions to \eqref{inverse-pl-s} and \eqref{hard-sphere} respectively. Then for any $T>0$, there exists a constant $C$ depending on $k,\d,|f_{\textup{in}}|_{L^1_{k+3+\delta}},|f_{\textup{in}}|_{W^{1,1}_{k+2+\delta}},|f_{\textup{in}}|_{L\log L},$ and $T$, independent of $s$, such that
    \ben\label{error-result}
\sup\limits_{t\in[0,T]}|F^s(t)|_{L^1_k}\le C(k,\delta,|f_{\textup{in}}|_{L^1_{k+3+\delta}},|f_{\textup{in}}|_{W^{1,1}_{k+2+\delta}},|f_{\textup{in}}|_{L\log L},T).
    \een
    
\end{thm}

\begin{rmk}[Optimality of the solution convergence]\label{optimality-solution}
The $O(s)$ convergence in Theorem~\ref{thm1.2} is also \textit{optimal} at the solution level.  
Indeed, Theorem~\ref{kernel-convergence} shows that the angular kernel admits a sharp
first-order expansion of the form
\[
b_s(\theta)=\frac14 + s\,\bar{b}(\theta) + o(s),
\qquad s\to0,
\]
with $\bar{b}(\theta)$ not identically zero.  
Since the collision operator $Q_s$ depends linearly on $b_s$, one obtains the exact expansion
\[
Q_s(f,f)=Q_0(f,f)+s\,\bar{Q}(f,f)+o(s),
\]
and thus $Q_s(f,f)-Q_0(f,f)$ cannot decay faster than order $s$ in any weighted $L^1_k$ norm.
Because the homogeneous Boltzmann flow is Fréchet differentiable with respect to the collision
kernel for general initial data, the corresponding solutions satisfy
\[
f_s(t)=f_0(t)+sF(t)+o(s),
\]
where the profile $F(t)$ solves a nontrivial linearized equation.
Hence no rate better than $O(s)$ can hold at the level of solutions.
Therefore the convergence rate in Theorem~\ref{thm1.2} is sharp.
\end{rmk}

%By Theorem \ref{thm1.2}, we prove that for any $t\in[0,T]$, quantified by the $L^{1}_{k}$-norm,$$f^s(t) = f^0(t) + O(s).$$

%$$|f^s(t)-f^0(t)|_{L^1_{k}}\to 0, \quad s\to0.$$

Based on Theorem \ref{kernel-convergence},
we can also consider the convergence of solutions in other function spaces, for instance, in general Sobolev spaces. In this article, we refrain us to the weighted $L^1$ framework because it is closely related to the physical quantities. 
Indeed, the $L^1_2$ norm corresponds mass and energy of solutions to the Boltzmann equation, and thus provides the most relevant convergence of $f^s\to f^0$.
To the best of our knowledge, this type of result on \textit{optimal} convergence rate has not been previously available and may stimulate further developments, including potential extensions to the spatially inhomogeneous Boltzmann equation.

\subsection{Main difficulties}

The analysis of the limit $s\to0$ presents several difficulties that originate
from the implicit structure of the exact angular kernel $b_s(\theta)$ for inverse-power potentials and from a delicate interplay between geometry and singularity near the endpoint $y_s(\varphi)\to1$. 
Unlike the limit $s\to 1$ to the Coulomb potential studied in \cite{yang2025boltzmann}, where grazing collisions dominate, the hard-sphere limit $s\to0$ does not correspond to a concentration of small deflection angles. 
Instead, the difficulty arises from the degeneracy of the impact-parameter map $y_s(\varphi)$ near $\varphi=\pi/2$, together with the highly nonlinear $s$--dependence hidden in the implicit definition.

\subsubsection{Implicit structure and nested compositions.}
The angular kernel for inverse-power potentials can be written as
\[
b_s(\theta)
 = 2^{4s-1}\,\frac{1}{\sin\theta}\,
 \beta_s(y_s(\varphi))\,\beta_s'(y_s(\varphi))\,y_s'(\varphi),
\qquad
\beta_s(y)=\frac{y}{(1-y^2)^s},
\]
where $y_s(\varphi)$ is defined implicitly as the inverse function of
\[
\varphi_s(y)
 = y\int_0^1 \frac{dz}{\sqrt{g(s,y,z)}},
\qquad
g(s,y,z)=1 - z^{1/s} - y^2(z^2 - z^{1/s}),
\]
and $\theta$ and $\varphi$ satisfy $2\varphi+\theta=\pi$.
Thus $b_s(\theta)$ involves a fourfold composition of
\[
\beta_s,\ \beta_s',\ y_s=\varphi_s^{-1},\ y_s',
\quad\text{and the map }\theta\leftrightarrow\varphi.
\]
Since none of these are available in closed form for $s\neq0$, there is no simple direct comparison between $b_s(\theta)$ and its hard-sphere limit value $1/4$.  
The first main difficulty is to control these compositions uniformly in $s$ as $s\to0$.

\subsubsection{Endpoint degeneracy as $y_s(\varphi)\to1$.}
In the limit $s\to0$, one has $y_s(\varphi)\to y_0(\varphi)=\sin\varphi$.
Hence the singularity in the limit arises from the fact that
\[
y_s(\varphi)\to 1 \quad \text{as} \quad \varphi\to\frac{\pi}{2}.
\]
Controlling the quantities
\[
(1-y_s)^{-1}, \qquad (1-y_s^2)^{-1}, \qquad y_s'(\varphi)
\]
requires a detailed analysis of the degeneracy of function $g(s,y,z)$ near $z=1$ and $y=1$.  
We prove precise stability bounds,
\[
\bigl|\varphi_s(y)-\arcsin y\bigr|
 \lesssim s\,y(1-y^2)^{-1/2}, 
\qquad
\bigl|\varphi_s'(y)-(1-y^2)^{-1/2}\bigr|
 \lesssim s(1-y^2)^{-3/2},
\]
and then transfer these to $y_s(\varphi)$ and $y_s'(\varphi)$ via inverse-function estimates.
This step is technically delicate because the inversion amplifies small errors in $\varphi_s$ near the boundary $\varphi=\pi/2$.

\subsubsection{Nonlinear decomposition of the kernel error.}
To obtain the optimal bound
\[
\bigl|b_s(\theta)-\tfrac14\bigr|\lesssim s\,\theta^{-2-2s},
\]
we decompose (see \eqref{decomposition-into-three-terms}-\eqref{term-3} for details)
\[
I=\left|b_s(\theta)-\frac14 \right|\le  I_1 + I_2 + I_3,
\]
where each term isolates a structural contribution coming from:
\begin{itemize}
    \item[(i)] the prefactor  deviation $2^{4s}-1$,
    \item[(ii)] the deviation $y_s'(\varphi)-\cos\varphi$,
    \item[(iii)] the deviation 
    \[
    \beta_s(y_s(\varphi))\,\beta_s'(y_s(\varphi)) 
    \;-\;
    \beta_0(y_0(\varphi))\,\beta_0'(y_0(\varphi))
    = 
    \beta_s(y_s(\varphi))\,\beta_s'(y_s(\varphi)) - y_0(\varphi),
    \]
    since $\beta_0(y)=y$ and $\beta_0'(y)=1$.
\end{itemize}
All three contributions contain mixtures of the small parameter $s$ and the singular factor $(1-y_s)^{-1}$.
A major part of the analysis is to show that each term is bounded by $Cs\theta^{-2-2s}$, using the uniform-in-s estimates for $y_s$, $y_s'$, and the explicit structures of $\beta_s$ and $\beta_s'$.

\subsubsection{Passing from kernel convergence to solution convergence.}
To compare the solutions $f_s$ and $f_0$, we introduce the scaled difference $F^s(t,v)=\frac{f_s(t,v)-f_0(t,v)}{s}$.  
It is easy to check that $F^s$ satisfies the following equation.
\ben\label{error-equation}
\pa_t F^s=Q^s(f^s,F^s)+Q^s(F^s,f^0)+\frac{Q^s-Q^0}{s}(f^0,f^0).
\een
The equation \eqref{error-equation} for $F^s$ involves the difference of collision operators,
the linear terms containing $F^s$ and the Boltzmann operators $Q^s$ with the angular singularity  $ \theta^{-2-2s}$.  
Since the normalization by $1/s$ cancels the leading singularity, all remaining bounds must be uniform in $s$.
This requires a refined angular decomposition, Povzner-type moment inequalities, and uniform propagation of weighted moments for both $Q_s$ and $Q_0$.
\subsubsection{Large-data framework.}
We work with general (possibly large) initial data in weighted $L^1_k$ spaces preserving physical moments for some $k \ge 2$.  In this setting, obtaining nonlinear estimates that remain stable as $s \to 0$ requires delicate analysis.  Consequently, deriving the \textit{optimal} \(O(s)\) convergence rate becomes a genuinely challenging task.

\subsection{Outline of the paper}

The remainder of the article is devoted to the proofs of 
Theorems~\ref{kernel-convergence} and~\ref{thm1.2} in Section~\ref{sec 2} and~\ref{sec 3} respectively.

\section{Convergence rate estimate of the kernel}\label{sec 2}

In this section, we prove Theorem \ref{kernel-convergence} with full details.   The angular part
$b_s(\th)$ contains the implicit function $y_s$ which is the inverse function of $\vf_s$ given by the integral \eqref{varphi-given-by-integral} and \eqref{def-of-g}. We first give some basic facts about the two functions. The implicit relation between the impact parameter and the deviation angle is
classical in kinetic theory; see, for instance, Landau--Lifschitz~\cite{LandauLifschitz}
and the exposition of Cercignani~\cite{Cercignani1988}.  We follow the same geometric
setup and obtain quantitatively uniform-in-s estimates.

By direct computation,
\ben \label{derivative-varphi}
\vf'_s(y)=\int_{0}^{1}\frac{1-z^{1/s}}{(1- z^{1/s} - y^2(z^2-z^{1/s}))^{3/2}}\mathrm{d}z=\int_{0}^{1}\frac{1-z^{1/s}}{g^{3/2}}\mathrm{d}z>0.
\een
Therefore, $\vf_s$ is a strictly increasing function, and so $y_s$ is well-defined and also strictly increasing. It is easy to check that
\beno 
\varphi = \vf_s(y):  y \in [0,1] \to  \varphi \in [0,\frac{\pi}{2}], \quad 	
y = y_s(\varphi): \varphi \in [0,\frac{\pi}{2}] \to y \in [0,1] .
\eeno 
One can 
compare $\varphi_s(y)$ to its limit $\varphi_0(y) \colonequals \arcsin y$.  For $s\in(0,1], y \in (0,1)$, 
\ben \label{comparison-varphi}
\varphi_s\left(y\right) \colonequals   y \int_{0}^{1}  \frac{1}{\sqrt{1- z^{1/s} - y^2(z^2-z^{1/s})
}} \mathrm{d}z>y\int_{0}^{1}\frac{1}{\sqrt{1-y^{2}z^{2}}}\mathrm{d}z=\arcsin y \equalscolon \varphi_0(y).
\een
As a result,  for $s\in(0,1], y \in (0,1), \vf \in (0,\pi/2)$, 
\ben\label{relation of sin vf and y}
\varphi_s\left(y\right)>\arcsin y\iff \arcsin y_s(\vf)<\vf \iff y_s(\vf)<\sin\vf .
\een

%$$B^{s}(v-v_*,\sigma)=|v-v_*|^{1-4s}\frac{2^{4s}}{2}\frac{1}{\sin\th}\beta_s(y)\beta_s'(y)y_s'(\varphi)=|v-v_*|^{\gamma}b_s(\cos\th).\quad \gamma=1-4s,s\in(0,1).$$
%
%
%
%Suppose $\rho$ is the impact parameter , $\th$ is deviation angle, and $2\varphi+\th=\pi$, define
%$$\beta=\rho(\frac{|v-v_*|}{2})^{2s}$$
%Thanks to \cite{Jang2023-df}, if $y$ solves the equation
%\ben\label{1.1}
%(\frac{y}{\beta})^{1/s}+y^2-1=0.
%\een
%then, the $\vf$ is given by a function of $y$:

The authors in \cite{Jang2023-df} proved that the limits 
\beno 
\vf_s(y)\to \arcsin y, \quad \vf'_s(y)\to \frac{1}{\sqrt{1-y^2}}, \quad s\to 0,
\eeno 
 hold locally uniform in $y\in [0,1)$. We further prove the limits with 
explicit convergence rates.
\begin{lem} \label{diff-varphi-and-limit}
  For any  $0 \leq y <  1, 0<s<1$, it holds that
    \ben \label{convergence_rate_of_vf_s_1}
    |\vf_s(y)-\arcsin y|\leq 2sy (1-y^2)^{-1/2},
    \\ \label{convergence_rate_of_derivative_vf_s}
    |\vf'_s(y)-\frac{1}{\sqrt{1-y^2}}|\leq 2s (1-y^2)^{-3/2}.
    \een
\end{lem}
\begin{proof}[Proof of Lemma \ref{diff-varphi-and-limit}] Recalling \eqref{def-of-g},
	it is easy to check 
	\ben\label{estimate of g_1}
	(1-y^{2}z^2)(1-z^{1/s})\leq g(s,y,z) \leq (1-y^{2}z^2), \quad z\in[0,1],y\in[0,1], s\in(0,1].
	\een
	We will frequently use \eqref{estimate of g_1}.
     Recalling \eqref{comparison-varphi},
    we have
\beno
    \begin{aligned}
   0\le\vf_s(y)-\arcsin y =     &y\int_{0}^{1} \left(\frac{1}{\sqrt{g(s,y,z)}}-\frac{1}{\sqrt{1-y^{2}z^{2}}} \right)\mathrm{d}z \\
        = &
        y\int_{0}^{1}\frac{(1-y^{2}z^{2})-g(s,y,z)}{\sqrt{1-y^{2}z^{2}}\sqrt{g(s,y,z)}(\sqrt{g(s,y,z)}+\sqrt{1-y^{2}z^{2}})}\mathrm{d}z\\  \eqref{def-of-g}
        = & y(1-y^{2})\int_{0}^{1}\frac{z^{1/s}}{\sqrt{1-y^{2}z^{2}}\sqrt{g(s,y,z)}(\sqrt{g(s,y,z)}+\sqrt{1-y^{2}z^{2}})}\mathrm{d}z\\
        \eqref{estimate of g_1}\leq & y(1-y^2)\int_{0}^{1}\frac{z^{1/s}}{(1-y^{2}z^{2})\sqrt{g(s,y,z)}(\sqrt{1-z^{1/s}}+1)}\mathrm{d}z\\
        \leq &
        y\int_{0}^{1}\frac{z^{1/s}}{\sqrt{g(s,y,z)}}\mathrm{d}z.
    \end{aligned}
 \eeno
   Now  by \eqref{estimate of g_1}, using the change of variable $t=z^{1/s}$,
    $$ y\int_{0}^{1}\frac{z^{1/s}}{\sqrt{g(s,y,z)}}\mathrm{d}z\leq \frac{y}{\sqrt{1-y^2}}\int_{0}^{1}\frac{z^{1/s}}{\sqrt{1-z^{1/s}}}\mathrm{d}z=\frac{y}{\sqrt{1-y^2}}sB(s+1,\frac{1}{2})\le\frac{y}{\sqrt{1-y^2}}sB(1,\frac{1}{2}).$$
    Here $B(a,b) \colonequals \int_{0}^{1}  t^{a-1} (1-t)^{b-1} \mathrm{d} t$ is the Beta function.
    Noting that $B(1,\frac{1}{2})=2$, we get \eqref{convergence_rate_of_vf_s_1}. 
    
    We now prove \eqref{convergence_rate_of_derivative_vf_s} on the derivative $\vf'_s$. As 
    $\vf'_s(y)-\frac{1}{\sqrt{1-y^2}}$ has no definite sign, we need to consider its absolute value 
    $|\vf'_s(y)-\frac{1}{\sqrt{1-y^2}}|$. Indeed,  $\vf_s'$ is bigger near $y=0$ and smaller near $y=1$ than $\frac{1}{\sqrt{1-y^2}}$. For more details, see the proof of Lemma 
    \ref{y-near-1-asy}. We separately consider the two differences $\vf'_s(y)-\frac{1}{\sqrt{1-y^2}}$
    and $\frac{1}{\sqrt{1-y^2}}-\vf_s'(y)$. Using
    \ben\label{expression1}
    \int_{0}^{1}\frac{1}{(1-y^{2}z^{2})^{3/2}}\mathrm{d}z=\frac{1}{\sqrt{1-y^2}},
    \een
    and recalling \eqref{derivative-varphi}, 
    $$
    \begin{aligned}
        \vf'_s(y)-\frac{1}{\sqrt{1-y^2}}&= \int_{0}^{1} \left( \frac{1-z^{1/s}}{g^{3/2}(s,y,z)}-\frac{1}{(1-y^{2}z^2)^{3/2}} \right) \mathrm{d}z \\
     \eqref{estimate of g_1}    &\leq  \int_{0}^{1} \left( \frac{1-z^{1/s}}{(1-y^{2}z^2)^{3/2}(1-z^{1/s})^{3/2}}-\frac{1}{(1-y^{2}z^2)^{3/2}} \right) \mathrm{d}z \\
        & =
       \int_{0}^{1} \frac{1}{(1-y^2 z^2)^{3/2}} \left((1-z^{1/s})^{-1/2}-1\right)\mathrm{d}z\\
        & \leq
        \frac{1}{(1-y^2)^{3/2}}\int_{0}^{1} \left((1-z^{1/s})^{-1/2}-1\right)\mathrm{d}z.  
    \end{aligned}
    $$
    Using the change of variables $t=z^{1/s}$,
    $$\int_{0}^{1}\left((1-z^{1/s})^{-1/2}-1\right)\mathrm{d}z=s\int_{0}^{1}[(1-t)^{-1/2}-1]t^{s-1}\mathrm{d}t.$$
  The integral is uniformly bounded for $0<s \leq 1$ by simply analyzing the singularity at $t=0$ or $t=1$. Indeed, $\int_{0}^{1}[(1-t)^{-1/2}-1]t^{s-1}\mathrm{d}t\le \int_{0}^{1}[(1-t)^{-1/2}-1]t^{-1}\mathrm{d}t=1.38629436111989...\le 2$. Then we have
  \ben \label{direction-1}
    	\vf'_s(y)-\frac{1}{\sqrt{1-y^2}} \leq
    	\frac{2s}{(1-y^2)^{3/2}}.  
    \een

    On the other direction, using the upper bound in \eqref{estimate of g_1},
    \beno
    \begin{aligned}
        \frac{1}{\sqrt{1-y^2}}-\vf'_s(y)&=\int_{0}^{1}\frac{1}{(1-y^{2}z^{2})^{3/2}}-\frac{1-z^{1/s}}{g^{3/2}}\mathrm{d}z\\
        &\leq
        \int_{0}^{1}\frac{z^{1/s}}{(1-y^{2}z^{2})^{3/2}}\mathrm{d}z\\
        & \le
        \frac{1}{(1-y^2)^{3/2}}\int_{0}^{1}z^{1/s}\mathrm{d}z = \frac{s}{s+1} \frac{1}{(1-y^2)^{3/2}}
        & \leq \frac{s}{(1-y^2)^{3/2}},
    \end{aligned}
\eeno
From this together with \eqref{direction-1}, we have \eqref{convergence_rate_of_derivative_vf_s}.
\end{proof}

In this article, we often exchange $1-y$ and $1-y^2$ as
$\f12 (1-y^2) \leq 1-y \leq 1-y^2$. 

The authors in \cite{Jang2023-df} proved that the following limits
$$y_s(\vf)\to \sin\vf, \quad y'_s(\vf)\to \cos \vf, \quad s\to 0,$$
hold locally uniform for $\vf\in[0,\frac{\pi}{2})$. We further prove that
\begin{lem} \label{diff-y-and-limit} For any  $0 \leq \vf \leq  \pi/2, 0<s<1$, it holds that
    \begin{align}
    \label{convergence_rate_of_y_s}
    |y_s(\vf)-\sin\vf| &\leq 2ys.
\\
 \label{convergence_rate_of_y'_s}
    |y'_s(\vf)-\cos\vf|&\leq 6s (1-y_s^2(\vf))^{-1} y'_s(\vf).
    \end{align}

\end{lem}
\begin{rmk}
    In the \cite{Jang2023-df}, the convergence of $y_s(\vf)\to \sin\vf$ is proved locally uniform for $\vf\in[0,\frac{\pi}{2})$ by using some analytic property of the involved functions. Here we will show that the convergence \eqref{convergence_rate_of_y_s} is uniform for $\vf\in[0.\frac{\pi}{2}]$ by direct computation.
\end{rmk}
\begin{proof}[Proof of Lemma \ref{diff-y-and-limit}]
Let $\vf \in (0,\frac{\pi}{2})$, 
since $y_s(\vf)\in (0,1)$, we can rewrite $y_s(\vf)=\sin(\arcsin(y_s(\vf)))$.
    $$
    \begin{aligned}
        |y_s(\vf)-\sin\vf|
        \leq & 
        |\sin(\arcsin y_s(\vf))-\sin(\vf)|\\
        = &
        |\cos w||\arcsin y_s(\vf)-\vf_s(y_s(\vf))|.
    \end{aligned}
    $$
    By \eqref{relation of sin vf and y}, $\arcsin y_s(\vf)<\vf$, $w\in(\arcsin y_s(\vf),\vf)$. Then $0\leq\cos w \leq \cos\arcsin y_s(\vf)=\sqrt{1-y_s(\vf)^2}$. Apply \eqref{convergence_rate_of_vf_s_1} by taking $y=y_s(\vf)$, we get the uniform convergence \eqref{convergence_rate_of_y_s}.

We now prove \eqref{convergence_rate_of_y'_s}. Recall $y_s$ is the inverse function of $\vf_s$, and so $\vf_s'(y)=(y'_s(\vf))^{-1}$. we make the rearrangement:
\beno
y'_s(\vf)-\cos\vf=y'_s(\vf) (1- \frac{\cos\vf}{y'_s(\vf)})  =y'_s(\vf) (1- \vf_s'(y)  \cos\vf_s(y)),
\eeno
Also recall $\vf_s'(y)$ has limit $(1-y^2)^{-1/2}$ as $s \to 0$.
\beno
	1- \vf_s'(y)  \cos\vf_s(y) = \left( 1- (1-y^2)^{-1/2} \cos\vf_s(y) \right) + \left((1-y^2)^{-1/2} - \vf_s'(y)\right)  \cos\vf_s(y) \equalscolon I_1 + I_2 .
\eeno
We first consider $I_2$.
From \eqref{relation of sin vf and y},  $ \sin\vf_s(y) > y$ and so $\cos\vf_s(y) \leq  (1-y^2)^{1/2}$. Now using \eqref{convergence_rate_of_derivative_vf_s}, we have
\beno
| I_2| =  \Big|\left((1-y^2)^{-1/2} - \vf_s'(y)\right)  \cos\vf_s(y)\Big| \leq  2s (1-y^2)^{-3/2} \times (1-y^2)^{1/2} = 2s (1-y^2)^{-1}.
\eeno
We then consider $I_1$. Let $A \colonequals \frac{1-\sin^2\vf_s(y) }{1-y^2} \leq 1$, then
\beno
0 \leq I_1 =   1- (1-y^2)^{-1/2} \cos\vf_s(y)  = 1- A^{1/2} \leq  1 - A = \frac{\sin^2\vf_s(y) -y^2}{1-y^2}
\leq 2 \frac{\sin \vf_s(y) -y}{1-y^2}.
\eeno
Now using \eqref{convergence_rate_of_y_s}, we have $0 \leq I_1 \leq 4s (1-y^2)^{-1}$.  Combining $I_1$ and $I_2$, we have \eqref{convergence_rate_of_y'_s}.
\end{proof}

 Observe $\th\to0\iff \vf\to\frac{\pi}{2}\iff y \to 1$, so both $\th^{-1}$ and $(1-y)^{-1}$ are singular factors. We now relate $(1-y)^{-1}$ to $\th^{-1}$ in the following lemma.
\begin{lem}\label{y-near-1-asy} For any  $0 < \theta \leq  \pi, 0<s<1$, it holds that
\ben\label{estimate of 1-y}
(1-y)^{-1} \leq \min\left\{4 s^{-\frac{1}{2}}, \pi^{2}\th^{-1}\right\} \theta^{-1}.
\een
\end{lem}
\begin{proof}[Proof of Lemma  \ref{y-near-1-asy}] For any 
    $\theta \in(0,\pi]$, denote the corresponding value $\varphi = \frac{\pi}{2}-\frac{\theta}{2} \in [0, \pi/2), y=y_s(\varphi) \in [0,1)$.
By mean value theorem,
$$\frac{1-y_s(\varphi)}{\frac{\pi}{2}-\varphi}=\frac{y_s\left(\frac{\pi}{2}\right)-y_s(\varphi)}{\frac{\pi}{2}-\varphi}=y_s^{\prime}(\tilde{\varphi})=\frac{1}{\varphi_s^{\prime}(\tilde{y})}.$$
for some $\tilde{\varphi} \in (\varphi,\pi/2), \tilde{y} \in (y,1)$. Recalling \eqref{derivative-varphi},
 it is easy to see that $\varphi_{s}^{\prime}(y)$ is a constant $\frac{\pi}{2}$ independent of $y$ if $s=\frac{1}{2}$;
 increasing in $y\in[0,1]$ for $s<\frac{1}{2}$; decreasing in $y\in[0,1]$ for $s>\frac{1}{2}$. That is,
\ben \label{s-small-derivetive-varphi}
    \vf_s'(0)\le \vf'_s(y)\le \vf_s'(1),\quad  y\in[0,1],s\in (0,\frac{1}{2}],\\
    \label{s-large-derivetive-varphi}
    \vf_s'(1)\le \vf'_s(y)\le \vf_s'(0), \quad  y\in[0,1],s\in[\frac{1}{2},1).
\een 
The value $\vf_s'(0)$ is bounded uniformly in $0<s<1$,
\ben \label{vf-s-prime-0}
1=\int_{0}^{1}1\mathrm{d}z\le\vf'_s(0)=\int_{0}^{1}(1-z^{1/s})^{-1/2}\mathrm{d}z\le\int_{0}^{1}(1-z)^{-\frac{1}{2}}\mathrm{d}z=2,\quad  s\in(0,1).
\een
Recall the Wallis integral $W_n=\int_{0}^{\pi/2}\sin^{n}x\mathrm{d}x=\frac{1}{2}B(\frac{n+1}{2},\frac{1}{2})$ enjoys the estimate  $1 \leq \sqrt{n} W_n \leq \sqrt{\frac{\pi}{2}}$ for $n \geq 1$.
Thanks to the calculation of $\vf'_s(1)$ in 
\cite{Jang2023-df}, we get
\ben \label{vf-s-prime-1}
\frac{1}{\sqrt{s}} \leq \vf'_s(1)=\int_{0}^{1}\frac{1-z^{1/s}}{(1-z^2)^{3/2}}\mathrm{d}z=\frac{1}{s}W_{\frac{1}{s}} \leq  \sqrt{\frac{\pi}{2}} \frac{1}{\sqrt{s}}.
\een
Therefore, we can derive that in the case of $s\in(0,\frac{1}{2})$,
$$\frac{1-y_s(\varphi)}{\frac{\pi}{2}-\varphi}=\frac{1}{\varphi_s^{\prime}(\tilde{y})}\ge\frac{1}{\vf_s'(1)}
\quad \Longrightarrow \quad
 (1-y)^{-1}\le (\frac{\th}{2})^{-1}\vf_s'(1)\le \sqrt{2\pi} s^{-1/2}\th^{-1},$$
and in the case of $s\in(\frac{1}{2},1)$,
$$
\frac{1-y_s(\varphi)}{\frac{\pi}{2}-\varphi}=\frac{1}{\varphi_s^{\prime}(\tilde{y})}\ge\frac{1}{\vf_s'(0)} \quad \Longrightarrow \quad
(1-y)^{-1}\le (\frac{\th}{2})^{-1}\vf_s'(0)\le4\th^{-1}\le 4 s^{-1/2}\th^{-1}.
$$

By \eqref{relation of sin vf and y}, we have 
$$
1-y>1-\sin \vf=1-\cos \frac{\th}{2} = 2\sin^2\frac{\th}{4} \geq \frac{\th^2}{\pi^2} \quad \Longrightarrow \quad  (1-y)^{-1}\le  \pi^2 \th^{-2}.
$$
Here, we use the basic inequality $\sin x\ge \frac{2x}{\pi},x\in[0,\frac{\pi}{4}]$. Patching together the above estimates, we obtain \eqref{estimate of 1-y}.
\end{proof}

Using the identity $y'_s(\vf)=\frac{1}{\vf'_s(y_s(\vf))}$, we can also get a uniform bound of $y'_s(\vf)$ which will be used later on. From \eqref{s-small-derivetive-varphi}
and \eqref{s-large-derivetive-varphi},
$$
\begin{aligned}
	\vf_s'(1)^{-1}\le y'_s(\vf)\le \vf_s'(0)^{-1}, \quad  \vf\in[0,\frac{\pi}{2}],s\in(0,\frac{1}{2}],\\
	\vf_s'(0)^{-1}\le y'_s(\vf)\le \vf_s'(1)^{-1}, \quad  \vf\in[0,\frac{\pi}{2}],s\in[\frac{1}{2},1).
\end{aligned}
$$
Therefore, we have in the case of $s\in(0,\frac{1}{2})$, by \eqref{vf-s-prime-0},
$$y'_s(\vf)\le \vf_s'(0)^{-1}\le 1.$$
and in the case of $s\in[\frac{1}{2},1)$,  by \eqref{vf-s-prime-1},
$$y'_s(\vf)\le \vf_s'(1)^{-1}=sW_{\frac{1}{s}}^{-1}\le \sqrt{s} \leq 1.$$
Combining the two estimates, 
\ben\label{upper bound of y'}
y'_s(\vf)\le 1, \quad  \vf\in[0,\frac{\pi}{2}], s\in(0,1).
\een

With the good factor $y'_s(\vf)$, we can get an improved estimate on $(1-y)^{-1}$
in the following Lemma.

\begin{lem}\label{key-estimate} For any  $0 < \theta \leq  \pi, 0<s<1$, it holds that
\ben\label{estimate of 1-y 2}
(1-y)^{-1}y'_s(\vf)\le 18 \th^{-1}.
\een
\end{lem}
\begin{proof}[Proof of Lemma \ref{key-estimate}]
	To prove \eqref{estimate of 1-y 2}, it suffices to prove
	\ben \label{def-H-s-y}
	H_s(y) \colonequals \frac{\frac{\pi}{2}-\varphi_s(y)}{(1-y)\vf'_s(y)} \leq 9.
	\een
	If $s \geq 1/2$, then for some $\tilde{y} \in (y,1)$,	$$H_s(y)=\frac{\frac{\pi}{2}-\vf_s(y)}{(1-y)\vf'_s(y)}=\frac{\vf_s(1)-\vf_s(y)}{(1-y)\vf'_s(y)}=\frac{\vf'_s(\tilde{y})}{\vf'_s(y)}\le 1.$$
	Here in the last inequality, we used the fact that  $\vf'_s(y)$ is non-increasing in $y\in[0,1]$ for $s \geq \frac{1}{2}$ by recalling \eqref{derivative-varphi}.
From now on, we consider $s < 1/2$	and divide the following proof into two cases: $1-y^2\le s$ or $1-y^2\ge s$.

	\textbf{Case 1:} $1-y^2\le s$. Recalling \eqref{derivative-varphi},
	$$\vf'_s(1)-\vf'_s(y)=\int_{0}^{1}\left(\frac{1-z^{1/s}}{(1-z^2)^{3/2}}-\frac{1-z^{1/s}}{g^{3/2}(s,y,z)}\right)\mathrm{d}z=\int_{0}^{1}(1-z^{1/s})\left(\frac{g^{3/2}(s,y,z)-(1-z^2)^{3/2}}{g^{3/2}(s,y,z)(1-z^2)^{3/2}}\right)\mathrm{d}z$$
	Recalling \eqref{def-of-g} for the definition of $g(s,y,z)$,
	$$g^{3/2}(s,y,z)-(1-z^2)^{3/2}=g^{3/2}(s,t,z)|_{t=1}^{t=y} =3
	\int_{y}^{1}t(z^2-z^{1/s})g^{1/2}(s,t,z)\mathrm{d}t.$$
	We get the double integral:
	\beno
	\vf'_s(1)-\vf'_s(y)=\int_{0}^{1}\int_{y}^{1}\frac{3t(1-z^{1/s})(z^2-z^{1/s})g^{1/2}(s,t,z)}{g^{3/2}(s,y,z)(1-z^2)^{3/2}}\mathrm{d}t\mathrm{d}z
	\\ \le \int_{0}^{1}\int_{y}^{1} \frac{3t(1-z^{1/s})(z^2-z^{1/s})}{g(s,y,z)(1-z^2)^{3/2}}\mathrm{d}t\mathrm{d}z.
	\eeno
	Here in the last inequality, we use the fact $g(s,y,z)$ is decreasing in $y$ if $0<s<1/2$ and so $g(s,t,z)\le g(s,y,z)$. We now claim: if $1-y^2\le s$, then
	\ben\label{A1.2}
	\frac{z^2-z^{1/s}}{g(s,y,z)}\le \frac{1}{2s}.
	\een
	From this, we have	
	$$
	\begin{aligned}
		\vf'_s(1)-\vf'_s(y) &\le \frac{1}{2s}\int_{0}^{1}\int_{y}^{1}\frac{3t(1-z^{1/s})}{(1-z^2)^{3/2}}\mathrm{d}t\mathrm{d}z\\
		&= \frac{1}{2s}\int_{y}^{1}3t\mathrm{d}t\times \vf'_s(1)=\frac{3}{4}\frac{1-y^2}{s}\vf'_s(1) 
		\le \frac{3}{4}\vf'_s(1).
	\end{aligned}
	$$
	From this and \eqref{s-small-derivetive-varphi}, we get
	$$H_s(y)=\frac{\frac{\pi}{2}-\vf_s(y)}{(1-y)\vf'_s(y)}=\frac{\vf'_s(\tilde{y})}{\vf'_s(y)}\le \frac{\vf'_s(1)}{\vf'_s(y)}\le 4.$$

We now prove the claim \eqref{A1.2}. Recalling \eqref{def-of-g} for the definition of $g(s,y,z)$,
		it suffices to prove
		$$
		2s(z^2-z^{1/s})\le 1-z^{1/s}-y^2(z^2-z^{1/s}).
		$$
	 For fixed $0 <y< 1, 0<s<1/2$,	let us define
		\beno
		F(z) \colonequals 1-z^{1/s}-(2s+y^2)(z^2-z^{1/s}).
		\eeno 
		Then 	it suffices to prove $F(z) \geq 0$. Compute the derivative,
		$$\frac{\pa F}{\pa z}=-\frac{1}{s}z^{1/s-1}-(2s+y^2)2z+(2s+y^2)\frac{1}{s}z^{1/s-1}=(-4s-2y^2)z+
		(2+s^{-1}(y^2-1))z^{1/s-1}.$$
		If  $1-y^2\le s$, then
		 $4s+2y^2 \geq 4s + 2(1-s) = 2+2s \geq 2$. Also noting that
	 $2+s^{-1}(y^2-1) \leq 2$, we have
		$$\frac{\pa F}{\pa z}\le -2z+2z^{1/s-1}\le 0,
		\quad \Longrightarrow \quad F(z)\ge F(0)=0.$$

	\textbf{Case 2:} $1-y^2> s$. We first prove for $y \in [0,1]$ that
		\ben\label{arcsiny-ynear1}
	\frac{\pi}{2}-\vf_s(y)\le 2\frac{1-y}{\sqrt{1-y^2}}.
	\een
	Thanks to \eqref{relation of sin vf and y}, we have
	$\frac{\pi}{2}-\vf_s(y)\le \frac{\pi}{2}-\arcsin y$, and so it suffices to prove
	\ben\label{A.3}
	\frac{\pi}{2}-\arcsin y\le 2\frac{1-y}{\sqrt{1-y^2}}, \quad  y \in [0,1].
	\een
	With $y=\sin \alpha = \cos \beta, \alpha,\beta \in[0,\pi/2]$,  
	$$\frac{\pi}{2}-\arcsin y\le 2\frac{1-y}{\sqrt{1-y^2}}\iff \frac{\pi}{2}-\al\le 2\frac{1-\sin\al}{\cos\al}\iff \beta\le \frac{2(1-\cos\beta)}{\sin\beta}.$$
	We now prove the last inequality for $\beta\in[0,\frac{\pi}{2}]$.
	Let us define
	$h(\beta)=\frac{\beta\sin\beta}{1-\cos\beta}$. 	Then it suffices to prove $h(\beta) \leq 2$. 
	It is easy to check
	$$h'(\beta)=\frac{\sin\beta-\beta}{1-\cos\beta}<0, 	\quad \Longrightarrow \quad  h(\beta)\le h(0)=\lim\limits_{\beta\to0}\frac{\beta\sin\beta}{1-\cos\beta}=2.$$

 We now claim: if $1-y^2\ge s$, then
	\ben \label{varphi-prime-lower-bound}
	\vf'_s(y)\ge \frac{2}{9}\frac{1}{\sqrt{1-y^2}}.
	\een
	From \eqref{arcsiny-ynear1} and \eqref{varphi-prime-lower-bound}, we obtain
	$$H_s(y)=\frac{\frac{\pi}{2}-\vf_s(y)}{(1-y)\vf'_s(y)}\le 9.$$
Collecting all the above upper bound estimates on $H_s(y)$, we arrive at \eqref{def-H-s-y}.
	
	Now it remains to prove \eqref{varphi-prime-lower-bound}.  
	Recalling \eqref{derivative-varphi} and \eqref{expression1},
		$$
		\begin{aligned}
			\frac{9}{2}\vf'_s(y)-\frac{1}{\sqrt{1-y^2}}&=\int_{0}^{1} g^{-3/2}(s,y,z) \left(\frac{9}{2}(1-z^{1/s})-\frac{g^{3/2}(s,y,z)}{(1-y^2z^2)^{3/2}} \right)\mathrm{d}z.
		\end{aligned}
		$$
		Let $G(s,y,z) \colonequals \frac{9}{2}(1-z^{1/s})-\left(\frac{g(s,y,z)}{1-y^2z^2}\right)^{3/2}$, then 
		 it suffices to prove $G(s,y,z) \geq 0$. We compute the partial derivative $\frac{\pa G}{\pa z}$ as follows:
		$$
		\begin{aligned}
			\frac{\pa G}{\pa z}&=-\frac{9}{2}\frac{1}{s}z^{1/s-1}-\frac{3}{2}\left(\frac{g(s,y,z)}{1-y^2z^2}\right)^{1/2}\times \frac{\pa}{\pa z}\left(\frac{g(s,y,z)}{1-y^2z^2}\right), \\
			\frac{\pa}{\pa z}\left(\frac{g(s,y,z)}{1-y^2z^2}\right)&=\frac{\pa}{\pa z}\left( 1-\frac{1-y^2}{1-y^2z^2}z^{1/s}  \right)\\
			&=-\left(\frac{1-y^2}{(1-y^2z^2)^2}2y^2z^{1/s+1}+\frac{1-y^2}{1-y^2z^2}\frac{1}{s}z^{1/s-1}\right).
		\end{aligned}
		$$
		Thus
		$$\frac{\pa G}{\pa z}=-\frac{9}{2}\frac{1}{s}z^{1/s-1}+\frac{3}{2}\left(\frac{g(s,y,z)}{1-y^2z^2}\right)^{1/2}\left(\frac{1-y^2}{(1-y^2z^2)^2}2y^2z^{1/s+1}+\frac{1-y^2}{1-y^2z^2}\frac{1}{s}z^{1/s-1}\right).$$
		Using the facts $\frac{g(s,y,z)}{1-y^2z^2}\le 1,\frac{1-y^2}{1-y^2z^2}\le 1$, noting the condition $1-y^2 \geq s$ gives
		$\frac{1}{1-y^2z^2}\le \frac{1}{1-y^2}\le \frac{1}{s}$, for $0<s<1/2$,
		we get
		$$\frac{\pa G}{\pa z}\le -\frac{9}{2}\frac{1}{s}z^{1/s-1}+3\frac{1}{s}z^{1/s+1}+\frac{3}{2}\frac{1}{s}z^{1/s-1}\le 0.$$
		Therefore $G(s,y,z)$ is non-increasing w.r.t. $z$, $G(s,y,z)\ge G(s,y,1)=0$, which
	ends the proof.
	\end{proof} 
	
	We are ready to prove the Theorem \ref{kernel-convergence}.

\begin{proof}[Proof of Theorem \ref{kernel-convergence}] In this proof, $0<s<1$.  We decompose
\ben \label{decomposition-into-three-terms}
I=\Big|b_s(\cos\th)-\frac{1}{4}\Big|=\Big|\frac{1}{2}2^{4s}\frac{\beta_s(y)\beta'_s(y)y'_s(\vf)}{\sin\th}-\frac{1}{4}\Big|\le I_1+I_2+I_3,
\een
where
\ben \label{term-1}
I_1 &=&\frac{1}{2}(2^{4s}-1)\Big|\frac{\beta_s(y)\beta'_s(y)y'_s(\varphi)}{\sin\th}\Big|,\\
\label{term-2}
I_2 &=&\Big|\frac{1}{2}\frac{\beta_s(y)\beta'_s(y)y'_s(\varphi)}{\sin\th}-\frac{1}{2}\frac{\beta_s(y)\beta'_s(y)\cos\vf}{\sin\th}\Big|,
\\ \label{term-3}
I_2 &=&\Big|\frac{1}{2}\frac{\beta_s(y)\beta'_s(y)\cos\vf}{\sin\th}-\frac{1}{4}\Big|=\frac{1}{4}\Big|\frac{\beta_s(y)\beta'_s(y)}{\sin\vf}-1\Big|.
\een 
For simplicity, let
\ben\label{aa2.23}
A_s(y)=2s(1-y^2)^{-1-s},\quad B^{s}(y)=(1-2s)(1-y^2)^{-s},\quad \beta'_s(y)=A_s(y)+B^{s}(y).
\een

In order to make the proof clear, we shall divide it into five steps.

\textbf{Step 1}: \emph{$\th\in(0,\frac{\pi}{2})$. Estimate of $I_1$.}
$$
I_1=\frac{1}{2}(2^{4s}-1)\Big|\frac{\beta_s(y)\beta'_s(y)y'_s(\varphi)}{\sin\th}\Big|=\frac{1}{2}(2^{4s}-1)\Big|\frac{\beta_s(y)A_s(y)y'_s(\varphi)}{\sin\th}\Big|+\frac{1}{2}(2^{4s}-1)\Big|\frac{\beta_s(y)B^{s}(y)y'_s(\varphi)}{\sin\th}\Big|=I_{11}+I_{12}.
$$ 
Recalling \eqref{function-beta-s-def}, we have
\ben\label{aa2.24}
|\beta_s(y)A_s(y)y'_s(\varphi)|=2sy(1-y^2)^{-1-2s}y'_s(\varphi).
\een
By \eqref{estimate of 1-y 2}, we have
\ben\label{aa2.25}
(1-y^2)^{-1}y'_s(\varphi)\le 18\th^{-1}.
\een
From \eqref{estimate of 1-y}, using inequality: $e^{-\frac{1}{e}}\le s^{s}\le 1, s\in[0,1]$, one can derive
\ben\label{aa2.26}
(1-y^2)^{-2s}\le 4^{2s}\th^{-2s}s^{-s}\le 16\th^{-2s}s^{-s}\le 16e^{\frac{1}{e}}\th^{-2s}.
\een
By \eqref{aa2.24}, \eqref{aa2.25} and \eqref{aa2.26},
\ben\label{aa2.27}
|\beta_s(y)A_s(y)y'_s(\varphi)|\le 2ys\times18\times16 e^{\frac{1}{e}}\th^{-1-2s}\le 576e^{\frac{1}{e}}s\th^{-1-2s}.
\een
By using the basic inequality
\ben\label{inequality}
1-x^{\al}\leq\al \ln\frac{1}{x},  \quad x\in(0,1], \al>0,
\een
We take $\al=4s,x=2^{-1}$,
\ben\label{aa2.29}
2^{4s}-1\le 2^{4s}4s\ln 2\le 45s.
\een
Note the fact 
\ben\label{aa2.30}
\frac{\th}{\sin\th}\le\frac{\pi}{2},\quad\th\in(0,\frac{\pi}{2}].
\een
From \eqref{aa2.27}, \eqref{aa2.29} and \eqref{aa2.30}, we have
\ben\label{estimate-of-I_11}
I_{11}=\frac{1}{2}(2^{4s}-1)\Big|\frac{\beta_s(y)A_s(y)y'_s(\varphi)}{\sin\th}\Big|\le 6480\pi e^{\frac{1}{e}}s^{2}\th^{-2-2s}\le6480\pi e^{\frac{1}{e}}s\th^{-2-2s}.
\een
By \eqref{function-beta-s-def}, \eqref{aa2.23} and \eqref{aa2.26},
\ben\label{aa2.32}
|\beta_s(y)B^{s}(y)|=|(1-2s)y(1-y^2)^{-2s}|\le 16e^{\frac{1}{e}}\th^{-2s}.
\een

Using \eqref{upper bound of y'}, \eqref{aa2.29}, \eqref{aa2.30} and \eqref{aa2.32}, we derive
\ben\label{estimate-of-I_12}
I_{12}=\frac{1}{2}(2^{4s}-1)\Big|\frac{\beta_s(y)B^{s}(y)y'_s(\varphi)}{\sin\th}\Big|\le 180\pi e^{\frac{1}{e}}s\th^{-1-2s}\le 180\pi^2e^{\frac{1}{e}}s\th^{-2-2s}.
\een
We patch together \eqref{estimate-of-I_11} and \eqref{estimate-of-I_12} to get
\ben\label{estimate of I_1-1}
I_1\le (6480+180\pi)\pi e^{\frac{1}{e}}s\th^{-2-2s}.
\een
\textbf{Step 2}: \emph{$\th\in[\frac{\pi}{2},\pi]$. Estimate of $I_1$.}

We first show that 
\ben\label{aa2.34}
\frac{y}{\sin\th}\le \frac{\pi}{4}.
\een
Recalling \eqref{varphi-given-by-integral} and \eqref{def-of-g}, using $g(s,y,z)=1-z^{1/s}-y^2(z^2-z^{1/s})\le1$, one can derive
\ben\label{aa2.35}
\frac{y}{\vf_s(y)}=\frac{1}{\int_{0}^{1}\frac{1}{\sqrt{g}}\mathrm{d}z}\le \frac{1}{\int_{0}^{1}\frac{1}{1}\mathrm{d}z} \le1.
\een
From \eqref{aa2.30} and \eqref{aa2.35},
\ben\label{bound of y/sinth}
\frac{y}{\sin\th}=\frac{y}{\vf_s(y)}\frac{\vf_s(y)}{\sin\th}=\frac{y}{\vf_s(y)}\frac{\vf}{\sin\th}=\frac{1}{2}\frac{y}{\vf_s(y)}\frac{2\vf}{\sin2\vf}\le \frac{1}{2}\times1\times\frac{\pi}{2}=\frac{\pi}{4}, \quad \th\in[\frac{\pi}{2},\pi].
\een
Moreover, we can also get the uniform upper bound of $(1-y^2)^{-1}$ in this case. Recalling \eqref{varphi-given-by-integral},
$$\vf_s(y)=y\int_{0}^{1}\frac{1}{\sqrt{g(s,y,z)}}\mathrm{d}z,\quad g(s,y,z)=1-y^2z^2-(1-y^2)z^{1/s}.$$
Obviously, for $0<s_1<s_2<1$, we have
$$g(s_1,y,z)>g(s_2,y,z)\iff \vf_{s_1}(y)<\vf_{s_2}(y)\iff y_{s_2}(\vf)<y_{s_1}(\vf).$$
By the monotonicity, $y_s(\vf)\le y_s(\frac{\pi}{4})\le y_0(\frac{\pi}{4})=\sin \frac{\pi}{4}=\frac{\sqrt{2}}{2}.$
Therefore,
\ben\label{aa2.37}
(1-y^2)^{-1}\le(1-y^2)^{-1}|_{y=\frac{\sqrt{2}}{2}}= 2,\quad \th\in[\frac{\pi}{2},\pi].
\een
Using \eqref{aa2.37} and \eqref{bound of y/sinth}, we have
\ben\label{aa2.38}
\frac{\beta_s(y)A_s(y)}{\sin\th}=2s\frac{y}{\sin\th}(1-y^2)^{-1-2s}\le 2s\times\frac{\pi}{4}\times2^{1+2s}\le 4\pi s,
\een
and
\ben\label{aa2.39}
\Big|\frac{\beta_s(y)B^{s}(y)}{\sin\th}\Big|=\Big|(1-2s)\frac{y}{\sin\th}(1-y^2)^{-2s}\Big|\le\frac{\pi}{4}\times 2^{2s}\le \pi.
\een
From \eqref{upper bound of y'}, \eqref{aa2.29} and \eqref{aa2.38},
$$I_{11}=\frac{1}{2}(2^{4s}-1)\Big|\frac{\beta_s(y)A_s(y)y'_s(\varphi)}{\sin\th}\Big|\le \frac{1}{2}(2^{4s}-1)|y_s'| \times4\pi s\le 90\pi s^{2}.$$
From \eqref{upper bound of y'}, \eqref{aa2.29} and \eqref{aa2.39},
$$I_{12}=\frac{1}{2}(2^{4s}-1)\Big|\frac{\beta_s(y)B^{s}(y)y'_s(\varphi)}{\sin\th}\Big|\le \frac{1}{2}(2^{4s}-1)|y_s'|\times \pi\le 25\pi s.$$
Combine these, one can get the uniform convergence : $I_{1}\le 115\pi s.$ Since $\th\in[\frac{\pi}{2},\pi]$, we can rewrite 
\ben\label{estimate of I_1-2}
I_{1}\le 115\pi s\le 115\pi^{5} s\th^{-2-2s}.
\een

\textbf{Step 3}: \emph{$\th\in(0,\frac{\pi}{2})$. Estimate of $I_2$.}

$$I_2\le \Big|\frac{1}{2}\frac{\beta_s(y) A_s(y)y'_s}{\sin\th}\Big|+\Big|\frac{1}{2}\frac{\beta_s(y) A_s(y)\cos\vf}{\sin\th}\Big|+\frac{1}{2}\Big|\frac{\beta_s(y)B^{s}(y)}{\sin\th}\Big|\times|y'_s-\cos\vf|\eqqcolon I_{21}+I_{22}+I_{23}.$$
By \eqref{aa2.27} and \eqref{aa2.30}, one can derive
\ben\label{estimate-of-I_21}
I_{21}\le \frac{1}{2}\frac{\pi}{2}\th^{-1}\times576e^{\frac{1}{e}}s\th^{-1-2s} \le 144\pi e^{\frac{1}{e}}s\th^{-2-2s}.
\een
Since $2\vf+\th=\pi$,
\ben\label{aa2.42}
I_{22}=\Big|\frac{1}{2}\frac{\beta_s A_s(y)\cos\vf}{\sin\th}\Big|=|\frac{1}{2}\frac{\beta_s A_s(y)\sin\frac{\th}{2}}{\sin\th}|\le \frac{1}{4\cos\frac{\th}{2}}|\beta_sA_s(y)|.
\een
Using \eqref{function-beta-s-def}, \eqref{aa2.23} and $\cos\th\ge \frac{\sqrt{2}}{2}\ge \frac{1}{2}$,
\ben\label{aa2.43} 
I_{22}\le\frac{1}{4\cos\frac{\th}{2}} |\beta_sA_s(y)|\le \frac{1}{2}\times2s(1-y^2)^{-2s}=s(1-y^2)^{-2s}.
\een
By \eqref{estimate of 1-y}, we have 
\ben\label{aa2.44}
(1-y^2)^{-1}\le \pi^{2}\th^{-2}. 
\een
Using \eqref{aa2.26}, \eqref{aa2.43} and \eqref{aa2.44}, we derive
\ben\label{estimate-of-I_22}
I_{22}\le s(1-y^2)^{-1-2s}=(1-y^2)^{-1}\times (1-y^2)^{-2s}s\le s\pi^2\th^{-2}\times 16e^{\frac{1}{e}}\th^{-2s}=16s\pi^2e^{\frac{1}{e}}\th^{-2-2s}.
\een
For $I_{23}$, by \eqref{aa2.30} and \eqref{aa2.32}, we have
\ben\label{aa2.46}
\Big|\frac{\beta_sB^{s}(y)}{\sin\th}\Big|\le  8\pi e^{\frac{1}{e}}\th^{-1-2s}.
\een
By \eqref{aa2.46}, \eqref{convergence_rate_of_y'_s} and \eqref{estimate of 1-y 2}, we get
\beno
I_{23}=\frac{1}{2}\times\Big|\frac{\beta_sB^{s}(y)}{\sin\th}\Big|\times |y_s'-\cos\vf|\le 4\pi e^{\frac{1}{e}}\th^{-1-2s}\times 6s(1-y^2)^{-1}y_s'(\vf)=432s\pi e^{\frac{1}{e}}\th^{-2-2s}.
\eeno
Therefore, 
\ben\label{estimate of I_2-1}
I_2\le (576+16\pi)\pi e^{\frac{1}{e}}s\th^{-2-2s}.
\een

\textbf{Step 4}: \emph{$\th\in[\frac{\pi}{2},\pi]$. Estimate of $I_2$.}

Recalling \eqref{upper bound of y'} and \eqref{aa2.38}, we derive
$$I_{21}+I_{22}=\Big|\frac{1}{2}\frac{\beta_s A_s(y)y'_s}{\sin\th}\Big|+\Big|\frac{1}{2}\frac{\beta_s A_s(y)\cos\vf}{\sin\th}\Big|\le\Big|\frac{\beta_s(y)A_s(y)}{\sin\th}\Big|\le 4\pi  s.$$
By \eqref{aa2.39}, we have
$$I_{23}=\Big|\frac{\beta_sB^{s}(y)}{\sin\th}\Big|\times|y'_s-\cos\vf|\le \pi\times|y'_s-\cos\vf|.$$
Using \eqref{convergence_rate_of_y'_s} again, we have
$$I_{23}\le 6\pi s(1-y^2)^{-1}y_s'(\vf).$$
The results \eqref{aa2.37} and \eqref{upper bound of y'} give the uniform convergence
$$I_{23}\le 12\pi s.$$
Therefore,
\ben\label{estimate of I_2-2}
I_{2}\le 16\pi s \le 16\pi^5 s\th^{-2-2s},\quad \th\in[\frac{\pi}{2},\pi].
\een

\textbf{Step 5}: \emph{$\th\in(0,\pi]$. Estimate of $I_3$.}

Using $2\vf+\th=\pi$, one can derive
$$I_3=\left|\frac{1}{2}\frac{\beta'_s\beta_s \cos\varphi}{\sin\th}-\frac{1}{4}\right|=\frac{1}{2}\Big|\frac{\beta'_s\beta_s \cos\varphi}{\sin\th}-\frac{1}{2}\Big|=\frac{1}{2}\Big|\frac{\beta'_s\beta_s \cos\varphi}{2\sin\vf\cos\vf}-\frac{1}{2}\Big|=\frac{1}{4}\left|\frac{\beta'_s\beta_s}{\sin\varphi}-1\right|.$$
Therefore,
$$I_3\le \frac{1}{4}(I_{31}+I_{32}+I_{33}),$$
 where
$$I_{31}=\left|\frac{\beta_s \beta_s^{\prime}}{\sin \varphi}-\frac{y \beta_s^{\prime}}{\sin \varphi}\right|, \quad I_{32}=\left|\frac{y \beta_s^{\prime}}{\sin \varphi}-\frac{y}{\sin \varphi}\right|,\quad I_{33}=\left|\frac{y}{\sin \varphi}-1\right|.$$
Using \eqref{aa2.30} and \eqref{aa2.35}, we have
\ben\label{aa2.49}
\frac{y}{\sin\vf}=\frac{y}{\vf}\frac{\vf}{\sin\vf}\le \frac{\pi}{2}.
\een
Note that
$$I_{31}=\left|\frac{\beta_s\beta'_s}{\sin\varphi}-\frac{y\beta'_s}{\sin\varphi}\right|= \frac{|\beta'_s|}{\sin\vf}|\beta_s-y|.$$
Since $|\beta_s-y|=y|(1-y^2)^{-s}-1|=y(1-y^2)^{-s}|1-(1-y^2)^{s}|.$ Using the basic inequality \eqref{inequality} again, one can take $x=(1-y^2),\al=s$ to derive
$$|1-(1-y^2)^{s}|\le s\ln (1-y^2)^{-1}\le s(1-y^2)^{-1}.$$
Therefore,
\ben\label{aa2.50}
|\beta_s-y|=y(1-y^2)^{-s}|1-(1-y^2)^s|\le sy(1-y^2)^{-1-s}.
\een
From \eqref{aa2.49} and \eqref{aa2.50},
\ben\label{aa2.51}
I_{31}= \frac{|\beta'_s|}{\sin\vf}|\beta_s-y|\le |\beta'_s|\Big|\frac{y}{\sin\vf}s(1-y^2)^{-1-s}\le \frac{\pi}{2}|\beta'_s|s(1-y^2)^{-1-s}.
\een
Recalling the decomposition \eqref{aa2.23},
\ben\label{aa2.52}
s|\beta'_s|(1-y^2)^{-1-s}\le A_s(y)s(1-y^2)^{-1-s}+|B^{s}(y)|s(1-y^2)^{-1-s}.
\een
By \eqref{aa2.23},
\ben\label{aa2.53}
A_s(y)s(1-y^2)^{-1-s}=2s^{2}(1-y^2)^{-2-2s}.
\een
From \eqref{estimate of 1-y},
\ben\label{2.45}
(1-y^2)^{-2-2s}\le 4^{2+2s}(s^{-\frac{1}{2}}\th^{-1})^{2+2s}=4^{2+2s}s^{-1-s}\th^{-2-2s}\le 64e^{\frac{1}{e}}s^{-1}\th^{-2-2s}.
\een
Plugging \eqref{2.45} into \eqref{aa2.53}, we have
\ben\label{aa2.55}
A_s(y)s(1-y^2)^{-1-s}\le 64e^{\frac{1}{e}}s^{2}s^{-1}\th^{-2-2s}=64e^{\frac{1}{e}}s\th^{-2-2s}.
\een
From \eqref{aa2.23} and \eqref{estimate-of-I_22},
\ben\label{aa2.56}
|B^{s}(y)|s(1-y^2)^{-1-s}=|(1-2s)|s(1-y^2)^{-1-2s}\le s(1-y^2)^{-1-2s}\le 16\pi^2 e^{\frac{1}{e}}s\th^{-2-2s}.
\een
From \eqref{aa2.51}, \eqref{aa2.52}, \eqref{aa2.55} and \eqref{aa2.56}, one can derive
\ben\label{aa2.57}
I_{31}\le \frac{\pi}{2}\times(64+16\pi^2)e^{\frac{1}{e}}s\th^{-2-2s}=(32+8\pi^2)\pi e^{\frac{1}{e}}s\th^{-2-2s}.
\een
From \eqref{aa2.49},
\ben\label{aa2.58}
I_{32}=\left|\frac{y \beta_s^{\prime}}{\sin \varphi}-\frac{y}{\sin \varphi}\right|\le\frac{\pi}{2}|\beta_s'-1|.
\een
By \eqref{function-beta-s-derivative-def},
\ben\label{aa2.59}
|\beta_s'-1|=A_s(y)-2s+B^{s}(y)-(1-2s)=2s\Big((1-y^2)^{-1-s}-1\Big)+|1-2s|\Big((1-y^2)^{-s}-1\Big).\een
From \eqref{aa2.50},
\ben\label{aa2.60}
|1-2s|\Big((1-y^2)^{-s}-1\Big)\le\Big((1-y^2)^{-s}-1\Big)\le  s(1-y^2)^{-1-s}.
\een
Using \eqref{aa2.59} and \eqref{aa2.60},
\ben\label{aa2.61}
|\beta_s'-1|\le 2s\Big((1-y^2)^{-1-s}-1\Big)+s(1-y^2)^{-1-s}= 3s(1-y^2)^{-1-s}-2s\le3s(1-y^2)^{-1-s}.
\een
Thanks to \eqref{estimate of 1-y}, one can derive
\ben\label{aa2.62}
(1-y^2)^{-1-s}\le (1-y)^{-1-s}\le \pi^{2+2s}\th^{-2-2s}\le 64\th^{-2-2s}.
\een
Using \eqref{aa2.58}, \eqref{aa2.61} and \eqref{aa2.62}, we can finally derive 
\ben\label{aa2.63}
I_{32}\le \frac{\pi}{2}|\beta_s'-1|\le \frac{\pi}{2}\times 64 \times3s\th^{-2-2s}=96\pi s\th^{-2-2s}.
\een
It remains to estimate $I_{33}=\Big|\frac{y}{\sin\vf}-1\Big|$. Note that
$$I_{33}=\frac{1}{\sin\vf}|y-\sin\vf|.$$
From\eqref{convergence_rate_of_y_s},
$$I_{33}=\frac{1}{\sin\vf}|y-\sin\vf|\le \frac{1}{\sin\vf}2ys=\frac{4\cos\vf y}{\sin2\vf}s\le \frac{4y}{\sin\th}s. $$
Therefore, we can derive,
in the case of $\th\in(0,\frac{\pi}{2})$,
$$\frac{\th}{\sin\th}\le\frac{\pi}{2}\implies\frac{4y}{\sin\th}s\le 2\pi s\th^{-1}\le \frac{1}{4}\pi^{4}s\th^{-2-2s},$$
and in the case of $\th\in[\frac{\pi}{2},\pi]$, by \eqref{aa2.34},
$$\frac{y}{\sin\th}\le\frac{\pi}{4}\implies\frac{4y}{\sin\th}s\le\pi s\le \pi^5s\th^{-2-2s}.$$
Therefore,
\ben\label{aa2.64}
I_{33}\le \pi^5s\th^{-2-2s}.
\een
Using \eqref{aa2.57}, \eqref{aa2.63} and \eqref{aa2.64}, we have
\ben\label{estimate of I_3}
I_3\le 1200s\th^{-2-2s}.
\een

Combining the above results \eqref{estimate of I_1-1}, \eqref{estimate of I_1-2}, \eqref{estimate of I_2-1}, \eqref{estimate of I_1-2} and \eqref{estimate of I_3}, for any $\th\in(0,\pi],s\in(0,1)$, by checking
$$\max\{(3240+180\pi)\pi e^{1/e},115\pi^5\}+\max\{(576+16\pi)\pi e^{\frac{1}{e}},16\pi^5\}+1200\le 50000,$$ we finish the proof.
\end{proof}
This completes the proof of Theorem \ref{kernel-convergence}, which establishes the optimal convergence of the kernel. In the next section, we prove the optimal convergence of solutions.

\section{Estimate in \texorpdfstring{$L^1_k$}{L1k} norm}\label{sec 3}
In this section, we will obtain the local-in-time upper bound estimate of the error function $F^s=\frac{f^s-f^0}{s}$. We recall that the spatially homogeneous Boltzmann equation with hard potentials
admits a unique global solution in weighted $L^1$ spaces; see
Mischler--Wennberg~\cite{MischlerWennberg1999} and Villani~\cite{Villani1998NewClassWeakSol} for both cutoff and non-cutoff regimes.
 Throughout this section, we consider $0<s<\frac{1}{8}$. By Lemma 4.1 in \cite{He2012}, one has the following uniform estimates for the propagation of the moment.
\begin{lem}\label{lem3.1}
    Suppose f is a solution of Boltzmann equation \eqref{inverse-pl-s} or \eqref{hard-sphere} with hard potential kernel with initial data $f_{\textup{in}}\in L^1_2\cap L\log L$. Then for $k\ge 2,\ t\ge 0$, there exists constant $C(k,|f_{\textup{in}}|_{L^1_2},|f_{\textup{in}}|_{L\log L})$, such that
    \beno
|f|_{L^1_k}\le C(k,|f_{\textup{in}}|_{L^1_2},|f_{\textup{in}}|_{L\log L})|f_{\textup{in}}|_{L^1_k}.
    \eeno
\end{lem}

We remark that $k\ge 4$ is asked in Lemma 4.1  in \cite{He2012} by using a key estimate given by Lemma 1 in \cite{Desvillettes2009-ky} stating 
\ben \label{order-1}
	|\int_{\mathbb{S}^2}(\langle v'\rangle^k-\langle v\rangle^k)b(\th)\mathrm{d}\sigma|\le C_k \int_{\mathbb{S}^2}b(\th)\sin\frac{\th}{2}\mathrm{d}\sigma|v-v_*|[\langle v\rangle^{k-1}+\langle v_*\rangle^{k-1}],\quad  k \ge 2,
	\\ \label{order-2}
	|\int_{\mathbb{S}^2}(\langle v'\rangle^k-\langle v\rangle^k)b(\th)\mathrm{d}\sigma|\le C_k \int_{\mathbb{S}^2}b(\th)\sin^2\frac{\th}{2}\mathrm{d}\sigma|v-v_*|^2[\langle v\rangle^{k-2}+\langle v_*\rangle^{k-2}],\quad k \ge 4.
\een
In order to consider the full range $0<s<1$ and cancel two-order singularity, He \cite{He2012} used 
\eqref{order-2} and so $k\ge 4$ is required. But we only consider small $s$ in the range $0<s<\frac{1}{8}$ and so one-order $\theta$ is enough, see \eqref{theta-integral} for more details. 
That's why we allow $k \geq 2$ in Lemma \ref{lem3.1}. 

%\begin{rmk}
%    The The reason is that a key estimate called Villani lemma (Lemma 1 in \cite{Desvillettes2009-ky}) is used in the proof for any $s\in[0,1)$, which is given by
%    In our paper, only $2\le k$ is needed, since the angular singularity of $b_s(\th)\sim \th^{-2-2s}$ is mild when $s$ is small. It can been checked in \eqref{theta-integral}.
%\end{rmk}
Motivated by Lemma 2.7 in \cite{CAO2022109641}, we have the following lemma.
\begin{lem}\label{lemma3.2}
    For any constant $k\ge2$, we have
    \beq\label{3.1}
    \begin{aligned}
&\langle v'\rangle^k-\cos^k\frac{\th}{2}\langle v\rangle^k-\sin^k\frac{\th}{2}\langle v_*\rangle^k\le C_k(\cos^{k-1}\frac{\th}{2}\langle v\rangle^{k-1}\sin\frac{\th}{2}\langle v_*\rangle+\cos\frac{\th}{2}\langle v\rangle\sin^{k-1}\frac{\th}{2}\langle v_*\rangle^{k-1}),\\
&\langle v'_*\rangle^k-\sin^k\frac{\th}{2}\langle v\rangle^k-\cos^k\frac{\th}{2}\langle v_*\rangle^k\le C_k(\cos^{k-1}\frac{\th}{2}\langle v_*\rangle^{k-1}\sin\frac{\th}{2}\langle v\rangle+\cos\frac{\th}{2}\langle v_*\rangle\sin^{k-1}\frac{\th}{2}\langle v\rangle^{k-1}).
\end{aligned}
    \eeq
\end{lem}
\begin{proof}[Proof of lemma \ref{lemma3.2}]
    Setting $\mathbf{k}=\frac{v-v_*}{|v-v_*|},\mathbf{y}=\frac{v \times v_*}{\left|v \times v_*\right|},\mathbf{x}=\mathbf{y} \times \mathbf{k},$ and $\cos\b=\frac{(v,v_*)}{|v||v_*|},\sigma=\mathbf{k} \cos \theta+\mathbf{x} \sin \theta \cos \varphi+\mathbf{y} \sin \theta \sin \varphi\in \mathbb{S}^2$, $\th\in[0,\pi]$, $\vf\in[0,2\pi]$.
    We will use the following representation of $|v'|$ and $|v'_*|$, which can be proved directly.
    \beno
|v'|^2=\cos^2\frac{\th}{2}|v|^2+\sin^2\frac{\th}{2}|v_*|^2+\sin\th\sin\beta\cos\vf|v||v_*|,\\
|v'_*|^2=\sin^2\frac{\th}{2}|v|^2+\cos^2\frac{\th}{2}|v_*|^2-\sin\th\sin\beta\cos\vf|v||v_*|.
    \eeno
    Thanks to the symmetric structure, we only prove the inequality involving $v'$.
    For the case $2\le k\le 4$, letting 
    $$
    F(t)=(\cos^2\frac{\th}{2}\langle v\rangle^2+tP)^{\frac{k}{2}},\quad P=\sin^2\frac{\th}{2}\langle v_*\rangle^{2}+\sin\th\sin\beta\cos\vf|v||v_*|.$$
    Since $0\le \frac{k}{2}-1\le1$,
    $$
    \begin{aligned}
    &\langle v'\rangle^k-\cos^k\frac{\th}{2}\langle v\rangle^k=F(1)-F(0)\\
    =&\frac{k}{2}P\int_0^1(\cos^2\frac{\th}{2}\langle v\rangle^2+tP)^{\frac{k}{2}-1}\mathrm{d}t\\
    \le & \frac{k}{2}P\int_0^1 \cos^{k-2}\frac{\th}{2}\langle v\rangle^{k-2}+t^{\frac{k}{2}-1}P^{\frac{k}{2}-1}\mathrm{d}t\\
    \le & \frac{k}{2}P\cos^{k-2}\frac{\th}{2}\langle v\rangle^{k-2}+P(\sin^{k-2}\frac{\th}{2}\langle v_*\rangle^{k-2}+(\sin\th\sin\beta\cos\vf|v||v_*|)^{k-2})\\
    \le & \sin^{k}\frac{\th}{2}\langle v_*\rangle^{k}+C_k(\cos^{k-1}\frac{\th}{2}\langle v\rangle^{k-1}\sin\frac{\th}{2}\langle v_*\rangle+\cos\frac{\th}{2}\langle v\rangle\sin^{k-1}\frac{\th}{2}\langle v_*\rangle^{k-1}).
    \end{aligned}
    $$
   Thus the case $2\le k\le 4$ is proved.
    
    For the case $k>4$, letting
    $$F(s,t)=(s\cos^2\frac{\th}{2}\langle v\rangle^2+t\sin^2\frac{\th}{2}\langle v_*\rangle^{2}+st\sin\th\sin\beta\cos\vf|v||v_*|)^{\frac{k}{2}}.$$
    Note that
    $$
    \langle v'\rangle^k-\cos^k\frac{\th}{2}\langle v\rangle^k-\sin^k\frac{\th}{2}\langle v_*\rangle=F(1,1)-F(1,0)-F(0,1)+F(0,0)
    =\int_0^1\int_0^1 \frac{\pa^2F(s,t)}{\pa s\pa t}\mathrm{d}s\mathrm{d}t.$$
    $$
    \begin{aligned}
    &\frac{\pa^2F(s,t)}{\pa s\pa t}=\frac{k}{2}(\frac{k}{2}-1)(s\cos^2\frac{\th}{2}\langle v\rangle^2+t\sin^2\frac{\th}{2}\langle v_*\rangle^{2}+st\sin\th\sin\beta\cos\vf|v||v_*|)^{\frac{k}{2}-2}\times\\
    & (\sin^2\frac{\th}{2}\langle v_*\rangle^2+s\sin\th\sin\beta\cos\varphi|v||v_*|)+\frac{k}{2}(s\cos^2\frac{\th}{2}\langle v\rangle^2+t\sin^2\frac{\th}{2}\langle v_*\rangle^{2}+s\sin\th\sin\beta\cos\vf|v||v_*|)^{\frac{k}{2}-1}\times\\
    &(\sin\th\sin\beta\cos\vf|v||v_*|).
    \end{aligned}
    $$
    Since $0< \frac{k}{2}-2$, there exists a constant $C_k$ such that
    $$(a+b)^{\frac{k}{2}-2}\le C_k (a^{\frac{k}{2}-2}+b^{\frac{k}{2}-2}).$$
    One has
    \beno
     \langle v'\rangle^k-\cos^k\frac{\th}{2}\langle v\rangle^k-\sin^k\frac{\th}{2}\langle v_*\rangle\le \int_0^1\int_0^1 \Big|\frac{\pa^2F(s,t)}{\pa s\pa t}\Big|\mathrm{d}s\mathrm{d}t \\ \le C_k(\cos^{k-1}\frac{\th}{2}\langle v\rangle^{k-1}\sin\frac{\th}{2}\langle v_*\rangle+\cos\frac{\th}{2}\langle v\rangle\sin^{k-1}\frac{\th}{2}\langle v_*\rangle^{k-1}),
    \eeno
   which ends the proof.
\end{proof}
In order to prove Theorem \ref{thm1.2}, we first prove the estimate for the propagation of $W^{1,1}_{k}$ norms.
\begin{lem}\label{lemma3.3}
Let $f^0$ be the solution of \eqref{hard-sphere} with initial data $f_{\textup{in}}\in W^{1,1}_2\cap L^1_{3}\cap L\log L$. 
    There exists  a constant $C(|f_{\textup{in}}|_{L^1_2},|f_{\textup{in}}|_{L\log L})>0$ such that
    \ben\label{3.3}
|f^0|_{W^{1,1}_2}\le |f_{\textup{in}}|_{W^{1,1}_2}e^{C(|f_{\textup{in}}|_{L^1_2},|f_{\textup{in}}|_{L\log L})|f_{\textup{in}}|_{L^1_3} t},\quad t\ge0.
    \een
    Moreover, let $2< k\in\mathbb{R}$, suppose $f_{\textup{in}}\in W^{1,1}_{k}\cap L^1_{k+1}\cap L\log L$,
    for any $T>0$, there exists a constant $C$, depending on $k,T,|f_{\textup{in}}|_{L^1_{k+1}},|f_{\textup{in}}|_{W^{1,1}_{k}},|f_{\textup{in}}|_{L\log L}$ such that
    \beno
|f^0|_{W^{1,1}_{k}}\le C(k,T,|f_{\textup{in}}|_{L^1_{k+1}},|f_{\textup{in}}|_{W^{1,1}_{k}},|f_{\textup{in}}|_{L\log L}),\quad 0\le t\le T.
    \eeno
\end{lem}
\begin{proof}[Proof of Lemma \ref{lemma3.3}]
    By \eqref{hard-sphere}, one has for any index $|\al|=1$,
    \beno
\pa_{t}\pa^{\al}f^0=Q^0(f^0,\pa^{\al}f^0)+Q^0(\pa^{\al}f^0,f^0).
    \eeno
    Recalling the  collision operator $Q^0$, 
    $$Q^0(g,f)(v)=\int_{\mathbb{R}^3\times \mathbb{S}^2}B^0(g'_*f'-g_*f)\mathrm{d}v_*\mathrm{d}\sigma.$$
    Therefore, for any smooth test function $\vf(v)$, one has
    \beq\label{3.6}
\begin{aligned}
    \langle Q^0(g,f)+Q^0(f,g),\vf\rangle=&\int B^0(g'_*f'-g_*f)\vf\mathrm{d}V+\int B^0(g'f'_*-gf_*)\vf\mathrm{d}V\\
    = & \int B^0 g_*f(\vf_*'+\vf'-\vf_*-\vf)\mathrm{d}V.
\end{aligned}
    \eeq
    Multiplying the equation \eqref{hard-sphere} with $\textup{\textup{sgn}}(\pa^{\al}f^0)\langle v\rangle^2$, integrating over $\mathbb{R}^3$, one has
    \beq\label{3.7}
    \begin{aligned}
        \frac{\mathrm{d}}{\mathrm{d}t}|\pa^{\al}f^0|_{L^1_2}=\langle Q^0(\pa^{\al}f^{0},f^0\rangle+Q^0(f^0,\pa^{\al}f^{0}),\textup{\textup{sgn}}(\pa^{\al}f^0)\langle v\rangle^2\rangle.
    \end{aligned}
    \eeq
    Note that
    \beq\label{3.8}
    \begin{aligned}
    &\langle Q^0(\pa^{\al}f^{0},f^0\rangle+Q^0(f^0,\pa^{\al}f^{0}),\textup{\textup{sgn}}(\pa^{\al}f^0)\langle v\rangle^2\rangle\\
    &\text{taking $g= f^0,f=\pa^\al f^0,\vf=\textup{\textup{sgn}}(\pa^\al f^0)\langle v\rangle^2$ in \eqref{3.6}}\\
    =&\int B^0f^0_*\pa^\al f^0(\textup{\textup{sgn}}(\pa^\al f^0)'_*\langle v'_*\rangle^2+\textup{sgn}(\pa^\al f^0)'\langle v'\rangle^2-\textup{sgn}(\pa^\al f^0)_*\langle v_*\rangle^2-\textup{sgn}(\pa^\al f^0)\langle v\rangle^2)\mathrm{d}V\\
    \le& \int B^0|\pa^\al f^0|f^0_*(\langle v'_*\rangle^2+\langle v'\rangle^2+\langle v_*\rangle^2-\langle v\rangle^2)\mathrm{d}V\\
    =& \int 2B^0|\pa^{\al}f^{0}|f^0_*\langle v_*\rangle^2\mathrm{d}V=2\pi\int_{\mathbb{R}^3\times\mathbb{R}^3}|v-v_*|f^0_*|\pa^{\al}f^0|\langle v_*\rangle^2\mathrm{d}v\mathrm{d}v_*.
    \end{aligned}
    \eeq
    By \eqref{3.7} and \eqref{3.8},
    \begin{equation}\label{3.9}
    \begin{aligned}
    \frac{\mathrm{d}}{\mathrm{d}t}|\pa^{\al}f^0|_{L^1_2}&\le 2\pi \int_{\mathbb{R}^3\times\mathbb{R}^3}|v-v_*||\pa^{\al}f^{0}|f^0_*\langle v_*\rangle^2\mathrm{d}v\mathrm{d}v_*\\
    &\le 2\pi\Big(|\pa^{\al}f^{0}|_{L^1_1}|f^0|_{L^1_2}+|\pa^{\al}f^{0}|_{L^1}|f^0|_{L^1_3} \Big)\\
    &\le 4\pi |\pa^{\al}f^{0}|_{L^1_2}|f^0|_{L^1_3}.
    \end{aligned}
    \end{equation}
    By Lemma \ref{lem3.1}, there exists a constant $C$ depending on $|f_{\textup{in}}|_{L^1_2},|f_{\textup{in}}|_{L\log L}$, such that for $t\ge0$,
    \ben\label{3.10}
|f^0|_{L^1_3}\le C(|f_{\textup{in}}|_{L^1_2},|f_{\textup{in}}|_{L\log L})|f_{\textup{in}}|_{L^1_3}.
    \een
    Combining \eqref{3.7}, \eqref{3.9} and \eqref{3.10}, we get
    \beno
\frac{\mathrm{d}}{\mathrm{d}t}|\pa^{\al}f^{0}|_{L^1_2}\le C(|f_{\textup{in}}|_{L^1_2},|f_{\textup{in}}|_{L\log L})|\pa^{\al}f^{0}|_{L^1_2}|f_{\textup{in}}|_{L^1_3}.
    \eeno
    Thanks to Gronwall's inequality, we derive for $t\ge 0$,
    \ben\label{3.12}
|\pa^{\al}f^{0}|_{L^1_2}\le |\pa^{\al}f_{\textup{in}}|_{L^1_2}e^{C(|f_{\textup{in}}|_{L^1_2},|f_{\textup{in}}|_{L\log L})|f_{\textup{in}}|_{L^1_3}t}.
    \een
    From \eqref{3.12}, 
    \beno
|f^0|_{W^{1,1}_2}= |f^0|_{L^1_2}+ \sum\limits_{|\al|=1}|\pa^{\al}f^{0}|_{L^1_2}\le |f_{\textup{in}}|_{W^{1,1}_2}e^{C(|f_{\textup{in}}|_{L^1_2},|f_{\textup{in}}|_{L\log L})|f_{\textup{in}}|_{L^1_3}t}.
    \eeno
    This completes the proof of \eqref{3.3}.\\ 
    Taking place $\textup{sgn}(\pa^{\al}f^0)\langle v\rangle^2$ with $\textup{sgn}(\pa^{\al}f^0)\langle v\rangle^{k}(2< k)$ and repeating similar argument, one has
    \ben\label{3.14}
    \frac{\mathrm{d}}{\mathrm{d}t}|\pa ^{\al}f^0|_{L^1_{k}}\le\int B^0|\pa^{\al}f^{0}|f^0_*(\langle v'_*\rangle^{k}+\langle v'\rangle^{k}+\langle v_*\rangle^{k}-\langle v\rangle^{k})\mathrm{d}V.
    \een

    By \eqref{3.1}, it gives
    \beq\label{3.15}
    \begin{aligned}
        &\langle v_*'\rangle^{k}+\langle v'\rangle^{k}+\langle v_*\rangle^{k}-\langle v\rangle^{k}\\
        \le& (\cos^{k}\frac{\th}{2}+\sin^{k}\frac{\th}{2}-1)\langle v\rangle^{k}+C_k(\cos^{k}\frac{\th}{2}+\sin^{k}\frac{\th}{2}+1)\langle v_*\rangle^{k}
        \\
        +&C_k(\langle v_*\rangle^{k-1}\langle v\rangle+\langle v_*\rangle\langle v\rangle^{k-1}).
    \end{aligned}
    \eeq
    From \eqref{3.14} and \eqref{3.15}, it reduces to estimate
    $$\int\limits_{\mathbb{R}^3\times\mathbb{R}^3\times\mathbb{S}^2}B^0|\pa^{\al}f^{0}|f^0_*(A_1+A_2+A_3)\mathrm{d}V,$$
    where 
    $$A_1=(\cos^{k}\frac{\th}{2}+\sin^{k}\frac{\th}{2}-1)\langle v\rangle^{k},\quad A_2=C_k(\cos^{k}\frac{\th}{2}+\sin^{k}\frac{\th}{2}+1)\langle v_*\rangle^{k},$$
    $$A_3=C_k(\langle v_*\rangle^{k-1}\langle v\rangle+\langle v_*\rangle\langle v\rangle^{k-1}).$$
   
  \textbf{Estimate of $A_1$.} First,
    \ben\label{3.16}
    \int B^0|\pa^{\al}f^0|f^0_* A_1 \mathrm{d}V=-C_1(k)\int_{\mathbb{R}^3\times\mathbb{R}^3}|v-v_*||\pa^{\al}f^0|f^0_*\langle v\rangle^{k}\mathrm{d}v\mathrm{d}v_*,
    \een
    where $C_1(k)=\int_{\mathbb{S}^2}\frac{1}{4}(1-\sin^{k}\frac{\th}{2}-\cos^{k}\frac{\th}{2})\mathrm{d}\sigma>0$ depending only on $k$.
 Let
   $\nu(v)=
\int_{\mathbb{R}^3}|v-v_*|f^0(v_*)\mathrm{d}v_*$. On the one hand, for any $K>0$,
\beq\label{3.17}
\begin{aligned}
    \nu(v)&\ge \int_{\mathbb{R}^3,|v-v_*|\ge K}|v-v_*|f^0(v_*)\mathrm{d}v_*\\
    &\ge K\int_{\mathbb{R}^3,|v-v_*|\ge K}f^0(v_*)\mathrm{d}v_*=K(|f^0|_{L^1}-\int_{\mathbb{R}^3,|v-v_*|\le K}f^0(v_*)\mathrm{d}v_*).
\end{aligned}
\eeq
By mass conservation, one has $|f^0|_{L^1}=|f_{\textup{in}}|_{L^1}$. Note that for any $R>1$,
\beq\label{estimate of convolution-1}
\begin{aligned}
\int_{\mathbb{R}^3,|v-v_*|\le K}&f^0(v_*)\mathrm{d}v_*\le \int_{|f^0|\ge R,|v-v_*|\le K}f^0(v_*)\mathrm{d}v_*+\int_{|f^0|\le R,|v-v_*|\le K}f^0(v_*)\mathrm{d}v_*\\
\le & \int_{|f^0|\ge R,|v-v_*|\le K}f^0(v_*)\log f^0(v_*)\mathrm{d}v_* (\log R)^{-1}+ \frac{4}{3}\pi K^3R\\
\le & \int_{|f^0|\ge 1}f^0(v_*)\log f^0(v_*)\mathrm{d}v_* (\log R)^{-1}+ \frac{4}{3}\pi K^3R\\
\le & (H(f^0)-\int_{|f^0|\le 1}f^0(v_*)\log f^0(v_*)\mathrm{d}v_*)(\log R)^{-1}+\frac{4}{3}\pi K^3R.
\end{aligned}
\eeq
By the well-known Boltzmann H-theorem, one has $H(f^0)\le H(f_{\textup{in}})$. Using a basic inequality,
$$x\log y-x\log x\le y,\quad y>0,x\ge 0.$$
One can take $x=f^0(v_*),y=e^{-|v_*|^2}$ to get
$$-\int_{|f^0|\le 1}f^0(v_*)\log f^0(v_*)\mathrm{d}v_*\le \int_{\mathbb{R}^3} f^0(v_*)|v_*|^2\mathrm{d}v_*+\int_{\mathbb{R}^3} e^{-|v_*|^2}\mathrm{d}v_*\le |f_{\textup{in}}|_{L^1_2}+\pi^{3/2}.$$
Therefore, \eqref{estimate of convolution-1} gives
\ben\label{estimate of convolution-2}
\int_{\mathbb{R}^3,|v-v_*|\le K}f^0(v_*)\mathrm{d}v_*\le (\log R)^{-1}(H(f_{\textup{in}})+|f_{\textup{in}}|_{L^1_2}+\pi^{3/2})+\frac{4}{3}\pi K^3R.
\een
From \eqref{3.17} and \eqref{estimate of convolution-2},
\beno
\nu(v)\ge K\Big(|f_{\textup{in}}|_{L^1}-(\log R)^{-1}(H(f_{\textup{in}})+|f_{\textup{in}}|_{L^1_2}+\pi^{3/2})-\frac{4}{3}\pi K^3R\Big).
\eeno
 By the assumption $f_{\textup{in}}\in L^1_{2}\cap L\log L$, one has $H(f_{\textup{in}})+|f_{\textup{in}}|_{L^1_{2}}<\infty$. We can take $R=\exp(\frac{4(H(f_{\textup{in}})+\pi^{3/2}+|f_{\textup{in}}|_{L^1_{2}})}{|f_{\textup{in}}|_{L^1}})$ and $K=(\frac{3}{16\pi}R^{-1}|f_{\textup{in}}|_{L^1})^{1/3}$. Then
\ben\label{3.21}
\nu(v)\ge \frac{K}{2}|f_{\textup{in}}|_{L^1}>0.
\een
On the other hand, it is easy to see
\beq\label{3.22}
\nu(v)\ge |v||f^0|_{L^1}-|f^0|_{L^1_1}\ge |v||f_{\textup{in}}|_{L^1}-|f_{\textup{in}}|_{L^1_2}.
\eeq
    By \eqref{3.21} and \eqref{3.22}, there exists a constant $C_2$ depending on $|f_{\textup{in}}|_{L^1},|f_{\textup{in}}|_{L^1_2},|f_{\textup{in}}|_{L\log L}$, such that 
    \ben\label{3.23}
\nu(v)\ge C_2\langle v \rangle .
    \een
    From \eqref{3.16} and \eqref{3.23},
\ben\label{3.24}
    \int B^0|\pa^{\al}f^0|f^0_*A_1\mathrm{d}V\le -C_1(k)C_2|\pa^{\al}f^0|_{L^1_{k+1}}.
\een

 \textbf{Estimate of $A_2$.} By Young's inequality,
$$
\begin{aligned}
\int B^0|\pa^{\al}f^0|f^0_* A_2\mathrm{d}V&= C_k\int_{\mathbb{S}^2}\frac{1}{4}(\cos^{k}\frac{\th}{2}+\sin^{k}\frac{\th}{2}+1)\mathrm{d}\sigma\int |v-v_*||\pa^{\al}f^0|f^0_*\langle v_*\rangle^{k}\mathrm{d}v\mathrm{d}v_*\\
&\le C_k\int_{\mathbb{R}^3\times\mathbb{R}^3}|\pa^{\al}f^0|f^0_*\langle v\rangle\langle v_*\rangle^{k}\mathrm{d}v\mathrm{d}v_*+
C_k\int_{\mathbb{R}^3\times\mathbb{R}^3}|\pa^{\al}f^0|f^0_*\langle v_*\rangle^{k+1}\mathrm{d}v\mathrm{d}v_*\\
&\le C_k\epsilon |\pa^{\al}f^0|_{L^1_{k+1}}|f^0|_{L^1}+C_k(\epsilon^{-\frac{1}{k}}+1)|\pa^{\al}f^0|_{L^1}|f^0|_{L^1_{k+1}}.
\end{aligned}
$$
Therefore,
\ben\label{3.25}
\int B^0|\pa^{\al}f^0|f^0_* A_2\mathrm{d}V\le C_k\epsilon |\pa^{\al}f^0|_{L^1_{k+1}}|f_{\textup{in}}|_{L^1}+C(k,\epsilon)|\pa^{\al}f^0|_{L^1}|f_{\textup{in}}|_{L^1_{k+1}}.
\een

 \textbf{Estimate of $A_3$.} By Young's inequality,
\beq\label{3.26}
\begin{aligned}
   & \int B^0|\pa^{\al}f^0|f^0_* A_3\mathrm{d}V=C_k\int B^0|\pa^{\al}f^0|f^0_*(\langle v\rangle\langle v_*\rangle^{k-1}+\langle v_*\rangle\langle v\rangle^{k-1})\mathrm{d}V\\
    &\le C_k\Big(\epsilon\int_{\mathbb{R}^3\times\mathbb{R}^3} |v-v_*||\pa^{\al}f^0|f^0_*\langle v\rangle^{k}\mathrm{d}v\mathrm{d}v_*+(\epsilon^{-\frac{1}{k-1}}+\epsilon^{-\frac{k}{k-1}})\int_{\mathbb{R}^3\times\mathbb{R}^3} |v-v_*||\pa^{\al}f^0|f^0_*\langle v_*\rangle^{k}\mathrm{d}v\mathrm{d}v_*\Big)\\
    &\le C_k\epsilon |\pa^{\al}f^0|_{L^1_{k+1}}|f_{\textup{in}}|_{L^1_2}+C_k(\epsilon^{-\frac{1}{k-1}}+\epsilon^{-\frac{k}{k-1}})|\pa^{\al}f^0|_{L^1_2}|f_{\textup{in}}|_{L^1_{k+1}}.
\end{aligned}
\eeq
Combining \eqref{3.24}, \eqref{3.25} and \eqref{3.26}, taking $\epsilon$ small enough, we derive
$$
\int B^0|\pa^{\al}f^0|f^0_*(A_1+A_2+A_3)\mathrm{d}V\le -\frac{1}{2}C_1(k)C_2|\pa^{\al}f^0|_{L^1_{k+1}}+C_3(k,|f_{\textup{in}}|_{L^1_2},|f_{\textup{in}|_{L\log L}})|\pa^{\al}f^0|_{L^1_2}|f_{\textup{in}}|_{L^1_{k+1}}.
$$
By \eqref{3.12},
\ben\label{3.27}
\frac{\mathrm{d}}{\mathrm{d}t}|\pa^{\al}f^0|_{L^1_{k}}\le -\frac{1}{2}C_1(k)C_2|\pa^{\al}f^0|_{L^1_{k+1}}+C_3|\pa^{\al}f_{\textup{in}}|_{L^1_2}e^{C|f_{\textup{in}}|_{L^1_3}t}|f_{\textup{in}}|_{L^1_{k+1}}.
\een
Solving \eqref{3.27},
\ben\label{3.28}
|\pa^{\al}f^0|_{L^1_{k}}\le |\pa^{\al}f_{\textup{in}}|_{L^1_{k}}\exp\{-\frac{1}{2}C_1C_2t\}+\frac{C_3|\pa^{\al}f_{\textup{in}}|_{L^1_{2}}|f_{\textup{in}}|_{L^1_{k+1}}}{C|f_{\textup{in}}|_{L^1_3}+\frac{1}{2}C_1C_2}\Big(\exp\{C|f_{\textup{in}}|_{L^1_3}t\}-\exp\{-\frac{1}{2}C_1C_2t\}\Big).
\een
For any $T>0$, \eqref{3.28} gives a bound of $|f^0|_{W^{1,1}_{k}}$, which depends on $k$, $T$, $|f_{\textup{in}}|_{W^{1,1}_{k}}$ , $|f_{\textup{in}}|_{L^1_{k+1}}$ and $|f_{\textup{in}}|_{L\log L}$.
\end{proof}
Now we are in a position to prove our Theorem \ref{thm1.2}. Without loss of generality, we may assume that $b_s(\th)$ is supported in the set $0\le \th\le\frac{\pi}{2}$, otherwise $b_s(\th)$ can be replaced by its symmetrized form 
$$\bar{b}_s(\th)=\left[b_s(\th)+b_s(\pi-\th)\right] \mathbf{1}_{0\le \th\le\frac{\pi}{2}} .$$
Note that by \eqref{convergence_rate_in_global}, for any $\th\in[0,\frac{\pi}{2}]$,
\ben\label{3.29}
|\bar{b}_s(\th)-\frac{1}{2}|\le |b_s(\th)-\frac{1}{4}|+|b_s(\pi-\th)-\frac{1}{4}|\le Cs\th^{-2-2s}+Cs(\pi-\th)^{-2-2s}\le Cs\th^{-2-2s}.
\een
For simplicity, we still use the notation $b_s(\th)$ to denote $\bar{b}_s(\th)$. As usual, we use $C$ to denote positive constant, which may differ across lines.
We emphasize that all constants appearing in the upper bound are independent of s. For any $0<s<\frac{1}{8}$,
$$\int_{\mathbb{S}^2}\sin^2\frac{\th}{2}b_s(\th)\mathrm{d}\sigma\le\int_{\mathbb{S}^2}\sin\frac{\th}{2}b_s(\th)\mathrm{d}\sigma\le C\int_{0}^{\frac{\pi}{2}}b_s(\th)\th^2\mathrm{d}\th\le C\Big(\int_{0}^{\frac{\pi}{2}}|b_s-\frac{1}{2}|\th^2\mathrm{d}\th+\int_{0}^{\frac{\pi}{2}}\frac{1}{2}\th^2\mathrm{d}\th\Big).$$
By \eqref{3.29},
$$\int_{0}^{\frac{\pi}{2}}|b_s-\frac{1}{2}|\th^2\mathrm{d}\th\le Cs\int_{0}^{\frac{\pi}{2}}\th^{-2s}\mathrm{d}\th=\frac{Cs}{1-2s}\Big(\frac{\pi}{2}\Big)^{1-2s}\le 
\frac{C}{6}\frac{\pi}{2}.$$
Therefore, there exists a constant $C$ independent of $s$, such that
\ben\label{theta-integral}
\int_{\mathbb{S}^2}\sin^2\frac{\th}{2}b_s(\th)\mathrm{d}\sigma+\int_{\mathbb{S}^2}\sin\frac{\th}{2}b_s(\th)\mathrm{d}\sigma\le C,\quad 0<s<\frac{1}{8}.
\een
\begin{proof}[Proof of Theorem \ref{thm1.2}] The error function $F^s=\frac{f^s-f^0}{s}$ satisfies \eqref{error-equation} which we recall below,
\ben\label{error-equation-2}
	\pa_t F^s=Q^s(f^s,F^s)+Q^s(F^s,f^0)+\frac{Q^s-Q^0}{s}(f^0,f^0) \colonequals \sum\limits_{i=1}^{3} A_i.
	\een
Multiplying \eqref{error-equation-2} with $\textup{sgn}(F^s)\langle v\rangle^k$, integrating over $\mathbb{R}^3$,
$$
\frac{\mathrm{d}}{\mathrm{d}t}|F^s|_{L^1_k}=\sum\limits_{i=1}^{3}\langle A_i,\textup{sgn}(F^s)\langle v\rangle^k\rangle.
$$
\beq\label{3.32}
\begin{aligned}
    &\langle A_1,\textup{sgn}(F^s)\langle v\rangle^k\rangle\\
    =& \langle Q^s(f^s,F^s),\textup{sgn}(F^s)\langle v\rangle^k\rangle\\
    =&\int B^s((f^s)'_*(F^s)'-f^s_*F^s)\textup{sgn}(F^s)\langle v\rangle^k\mathrm{d}V\\
    \le &\int B^s((f^s)'_*|F^s|'-f^s_*|F^s|)\langle v\rangle^k\mathrm{d}V\\
    = &\langle Q^s(f^s,|F^s|),\langle v\rangle^k\rangle\\
    =&\langle Q^s(f^0,|F^s|),\langle v\rangle^k\rangle+s\langle Q^s(F^s,|F^s|),\langle v\rangle^k\rangle\\
    \le& \langle Q^s(f^0,|F^s|),\langle v\rangle^k\rangle+s|\langle Q^s(F^s,|F^s|),\langle v\rangle^k\rangle|.
\end{aligned}
\eeq
To make the estimate clear, our proof is divided into four parts. Let
\beq\label{3.33}
\begin{aligned}
I_1=\langle Q^s(f^0,|F^s|),\langle v\rangle^k\rangle,&\quad I_2=\langle Q^s(F^s,f^0),\textup{sgn}(F^s)\langle v\rangle^k\rangle,\\
I_3=s|\langle Q^s(F^s,|F^s|),\langle v\rangle^k\rangle|,&\quad I_4=\langle \frac{Q^s-Q^0}{s}(f^0,f^0),\textup{sgn}(F^s)\langle v\rangle^k\rangle.
\end{aligned}
\eeq
By \eqref{error-equation-2}, \eqref{3.32} and \eqref{3.33},
\ben\label{3.34}
\frac{\mathrm{d}}{\mathrm{d}t}|F^s|_{L^1_2}\le \sum\limits_{i=1}^{4} I_i.
\een

\textbf{Estimate of $I_1$.} We have
$$
\begin{aligned}
    I_1&=\int B^s((f^0)'_*|F^s|'-f^0_*|F^s|)\langle v\rangle^k\mathrm{d}V=\int B^s f^0_*|F^s|(\langle v'\rangle^k-\langle v\rangle^k)\mathrm{d}V\\
    \eqref{3.1}\le &-\int B^s f^0_*|F^s|(1-\cos^{k}\frac{\th}{2})\langle v\rangle^k\mathrm{d}V\\
    &+C_k\int B^s f^0_*|F^s|(\sin^k\frac{\th}{2}\langle v_*\rangle^k+\sin^{k-1}\frac{\th}{2}\langle v_*\rangle^{k-1}\langle v\rangle+\sin\frac{\th}{2}\langle v_*\rangle\langle v\rangle^{k-1})\mathrm{d}V.
\end{aligned}
$$
Therefore,
\beq\label{3.36}
\begin{aligned}
I_1\le&-\int B^s f^0_*|F^s|(1-\cos^{k}\frac{\th}{2})\langle v\rangle^k\mathrm{d}V\\
&+C_k\int_{\mathbb{S}^2}\sin\frac{\th}{2}b_s(\th)\mathrm{d}\sigma\int_{\mathbb{R}^3\times\mathbb{R}^3} |v-v_*|^\gamma f^0_*|F^s|(\langle v_*\rangle^k+\langle v_*\rangle^{k-1}\langle v\rangle+\langle v_*\rangle\langle v\rangle^{k-1})\mathrm{d}v\mathrm{d}v_*\\
&\colonequals I_{11}+I_{12}.
\end{aligned}
\eeq
Note that by \eqref{theta-integral}, there exists a constant $C_k$ independent of $s$,
\ben\label{3.37}
I_{12}\le C_k|F^s|_{L^1_{k-1+\gamma}}|f^0|_{L^1_{k+\gamma}}.
\een
From \eqref{3.36} and \eqref{3.37},
\ben\label{3.38}
I_{1}\le -\int B^s f^0_*|F^s|(1-\cos^{k}\frac{\th}{2})\langle v\rangle^k\mathrm{d}V+C_k|F^s|_{L^1_{k-1+\gamma}}|f_{\textup{in}}|_{L^1_{k+\gamma}}.
\een

\textbf{Estimate of $I_2$.} Note that
\beq\label{3.39}
\begin{aligned}
    I_2&=\langle Q^s(F^s,f^0),\textup{sgn}(F^s)\langle v\rangle^k\rangle=\int B^s F^s_*f^0(\textup{sgn}(F^s)'\langle v'\rangle^k-\textup{sgn}(F^s)\langle v\rangle^k)\mathrm{d}V\\
    &=\int B^s F^s_*f^0 \textup{sgn}(F^s)'(\langle v'\rangle^k-\langle v\rangle^k)\mathrm{d}V+\int B^s F^s_*f^0 (\textup{sgn}(F^s)'-\textup{sgn}(F^s))\langle v\rangle^k\mathrm{d}V\\
    &\colonequals I_{21}+I_{22}.
\end{aligned}
\eeq
Note that
\beq\label{3.40}
\begin{aligned}
I_{21}=&\int B^s F^s_*f^0 \textup{sgn}(F^s)'(\langle v'\rangle^k-\langle v\rangle^k)\mathrm{d}V=\int B^s F^sf^0_* \textup{sgn}(F^s)'_*(\langle v_*'\rangle^k-\langle v_*\rangle^k)\mathrm{d}V\\
\eqref{3.1}\le& \int B^s |F^s|f^0_* \sin^k\frac{\th}{2}\langle v\rangle^k\mathrm{d}V+\int B^s |F^s|f^0_* (1-\cos^k\frac{\th}{2})\langle v_*\rangle^k\mathrm{d}V\\
&+C_k\int  B^s |F^s|f^0_*(\sin^{k-1}\frac{\th}{2}\langle v\rangle^{k-1}\langle v_*\rangle+\sin\frac{\th}{2}\langle v\rangle\langle v_*\rangle^{k-1})\mathrm{d}V
\\
\le & \int B^s |F^s|f^0_* \sin^k\frac{\th}{2}\langle v\rangle^k\mathrm{d}V+C_k|F^s|_{L^1_{k-1+\gamma}}|f_{\textup{in}}|_{L^1_{k+\gamma}}.
\end{aligned}
\eeq
Let 
\ben\label{3.41}
\mathcal{F}(t,v)\colonequals f^0(t,v)\langle v\rangle^k .
\een
By \eqref{3.41}, one has 
\beq\label{3.42}
\begin{aligned}
    I_{22}=&\int B^s F^s_*f^0 (\textup{sgn}(F^s)'-\textup{sgn}(F^s))\langle v\rangle^k\mathrm{d}V\\
    =& \int B^s F^s_* (\textup{sgn}(F^s)'\mathcal{F}'-\textup{sgn}(F^s)\mathcal{F})\mathrm{d}V-\int B^s F^s_* \textup{sgn}(F^s)'(\mathcal{F}'-\mathcal{F})\mathrm{d}V\\
    \colonequals &I_{221}-I_{222}.
\end{aligned}
\eeq
By cancellation lemma and lemma \ref{lem3.1},
\beq\label{3.43}
\begin{aligned}
I_{221}=&\int B^s F^s_* (\textup{sgn}(F^s)'\mathcal{F}'-\textup{sgn}(F^s)\mathcal{F})\mathrm{d}V\\
=&\int \Big(B^s(\frac{|v-v_*|}{\cos\frac{\th}{2}},\th)\frac{1}{\cos^3\frac{\th}{2}}- B^s(|v-v_*|,\th)\Big)F^s_* \textup{sgn}(F^s)\mathcal{F}\mathrm{d}V\\
\le & \int_{\mathbb{S}^2}\left(\cos^{-(3+\gamma)}\frac{\th}{2}-1\right)b_s(\th)\mathrm{d}\sigma\int_{\mathbb{R}^3\times\mathbb{R}^3}|v-v_*|^{\gamma}|F^s_*|f^0\langle v\rangle^k\mathrm{d}v\mathrm{d}v_*\\
\le & C_k|F^s|_{L^1_\gamma}|f_{\textup{in}}|_{L^1_{k+\gamma}}.
\end{aligned}
\eeq
Let $v(\kappa)=\kappa v'+(1-\kappa)v,\kappa\in[0,1]$, by Taylor formula, $\mathcal{F}'-\mathcal{F}=\int_{0}^{1}\grad\mathcal{F}(v(\kappa))\cdot (v'-v)\mathrm{d}\kappa$. Therefore,
$$
\begin{aligned}
    |I_{222}|=&\Big|\int B^s F^s_* \textup{sgn}(F^s)'(\mathcal{F}'-\mathcal{F})\mathrm{d}V\Big|\\
    \le & \int_{0}^{1}\int b_s(\th)\sin\frac{\th}{2}|F^s_*||v-v_*|^{1+\gamma}|\grad\mathcal{F}(\kappa)|\mathrm{d}V\mathrm{d}\kappa\\
    =& \int_{0}^{1}\int b_s(\th)\sin\frac{\th}{2}|F^s_*|\Big(\frac{|v-v_*|}{\psi_\kappa(\th)}\Big)^{1+\gamma}|\grad\mathcal{F}(v)|\frac{1}{\psi^3_\kappa(\th)}\mathrm{d}V\mathrm{d}\kappa.
\end{aligned}
$$
Here, we use the change of variables $v\to v(\kappa)$ and $\psi_\kappa(\th)=(\cos^2\frac{\th}{2}+(1-\kappa)^2\sin^2\frac{\th}{2})^{\frac{1}{2}}\in[\frac{\sqrt{2}}{2},1]$. One can check the details in Lemma 2.2 in \cite{He2024}. Therefore,
$$\Big|\int B^s F^s_* \textup{sgn}(F^s)'(\mathcal{F}'-\mathcal{F})\mathrm{d}V\Big|\le C|F^s|_{L^1_{1+\gamma}}|\grad\mathcal{F}|_{L^1_{1+\gamma}}.$$
Recalling \eqref{3.41}, $\grad\mathcal{F}=\grad f^0\langle v\rangle^k+kf^0\langle v\rangle^{k-2} v$, by lemma \ref{lem3.1},
\beq\label{3.45}
\begin{aligned}
    |I_{222}|\le C|F^s|_{L^1_{1+\gamma}}|\grad\mathcal{F}|_{L^1_{1+\gamma}}\le C_k(|F^s|_{L^1_{1+\gamma}}|\grad f^0|_{L^1_{k+1+\gamma}}+|F^s|_{L^1_{1+\gamma}}|f_{\textup{in}}|_{L^1_{k+\gamma}}).
\end{aligned}
\eeq
By \eqref{3.42}, \eqref{3.43} and \eqref{3.45},
\ben\label{3.46}
I_{22}\le C_k|F^s|_{L^1_{1+\gamma}}(|f_{\textup{in}}|_{L^1_{k+\gamma}}+|\grad f^0|_{L^1_{k+1+\gamma}}).
\een
From \eqref{3.39}, \eqref{3.40} and \eqref{3.46},
\ben\label{3.47}
I_2\le \int B^s |F^s|f^0_* \sin^k\frac{\th}{2}\langle v\rangle^k\mathrm{d}V+C_k|F^s|_{L^1_{k-1+\gamma}}(|f_{\textup{in}}|_{L^1_{k+\gamma}}+|\grad f^0|_{L^1_{k+1+\gamma}}).
\een
By \eqref{3.38} and \eqref{3.47},
$$
I_1+I_2\le -\int B^s f^0_*|F^s|((1-\cos^{k}\frac{\th}{2}-\sin^{k}\frac{\th}{2})\langle v\rangle^k\mathrm{d}V+C_k|F^s|_{L^1_{k-1+\gamma}}(|f_{\textup{in}}|_{L^1_{k+\gamma}}+|\grad f^0|_{L^1_{k+1+\gamma}}).
$$
Let $$
I_1(k)\colonequals\min\limits_{\th\in[0,\frac{\pi}{2}]}\frac{1-\cos^k\frac{\th}{2}-\sin^k\frac{\th}{2}}{ \th^2}\implies 1-\cos^k\frac{\th}{2}-\sin^k\frac{\th}{2}\ge I_1(k)\th^2.
$$
Recalling \eqref{3.21} and \eqref{3.22},
$$\int_{\mathbb{R}^3}|v-v_*|^{\gamma}f^0_*\mathrm{d}v_*\ge |f_{\textup{in}}|_{L^1}|v|^{\gamma}-|f_{\textup{in}}|_{L^1_2}.$$
There exists a constant $C$ independent of $s$ such that
$$\int_{\mathbb{R}^3}|v-v_*|^{\gamma}f^0_*\mathrm{d}v_*\ge C\langle v\rangle^\gamma.$$
There exists a constant $C_1(k)\sim I_1(k)$ independent of $s$ such that
\ben\label{3.48}
I_1+I_2\le -C_1(k)|F^s|_{L^1_{k+\gamma}}+C_k|F^s|_{L^1_{k-1+\gamma}}(|f_{\textup{in}}|_{L^1_{k+\gamma}}+|\grad f^0|_{L^1_{k+1+\gamma}}).
\een

\textbf{Estimate of $I_3$.} Recalling \eqref{3.33},
\beq\label{3.49}
    I_3=s|\langle Q^s(F^s,|F^s|),\langle v\rangle^k\rangle|=s\Bigg|\int B^s F^s_*|F^s|(\langle v'\rangle^k- \langle v\rangle^k)\mathrm{d}V\Bigg|.
\eeq
Using \eqref{3.1} again,
\beq\label{3.50}
\begin{aligned}
    &\Bigg|\int B^s F^s_*|F^s|(\langle v'\rangle^k- \langle v\rangle^k)\mathrm{d}V\Bigg|\\
    \le & \Bigg|\int B^sF^s_*|F^s|((\cos^k\frac{\th}{2}-1)\langle v\rangle^k+\sin^k\frac{\th}{2}\langle v_*\rangle^k+C_k(\sin\frac{\th}{2}\langle v\rangle^{k-1}\langle v_*\rangle+\sin^{k-1}\frac{\th}{2}\langle v\rangle\langle v_*\rangle^{k-1})\mathrm{d}V\Bigg|\\
    \le & C_k(|F^s|_{L^1_{k+\gamma}}|F^s|_{L^1_\gamma}+|F^s|_{L^1_{k-1+\gamma}}|F^s|_{L^1_{1+\gamma}})\\
    \le& C_k|F^s|_{L^1_{k+\gamma}}|F^s|_{L^1_{1+\gamma}}.
\end{aligned}
\eeq
By \eqref{3.49} and \eqref{3.50},
$$
I_3\le C_ks|F^s|_{L^1_{k+\gamma}}|F^s|_{L^1_{1+\gamma}}.
$$
Note that by lemma \ref{lem3.1},
$$s|F^s|_{L^1_{k+\gamma}}=|f^s-f^0|_{L^1_{k+\gamma}}\le C_k|f_{\textup{in}}|_{L^1_{k+\gamma}}.$$
Therefore,
\ben\label{3.51}
I_3\le C_k|f_{\textup{in}}|_{L^1_{k+\gamma}}|F^s|_{L^1_{1+\gamma}}.
\een

\textbf{Estimate of $I_4$.} Note that
$$I_{4}\colonequals\langle A_4,\textup{sgn}(F^s)\langle v \rangle^k\rangle=\langle \frac{Q^s-Q^0}{s}(f^0,f^0),\textup{sgn}(F^s)\langle v \rangle^k\rangle.$$
We decompose $\langle Q^s(f^0,f^0),\textup{sgn}(F^s)\langle v\rangle^k\rangle$  into two parts,
\beq\label{3.52}
\begin{aligned}
&\langle Q^s(f^0,f^0),\textup{sgn}(F^s)\langle v\rangle^k\rangle=\int B^s f^0_*f^0(\textup{sgn}(F^s)'\langle v'\rangle^k-\textup{sgn}(F^s)\langle v\rangle^k)\mathrm{d}V\\
=&\int B^sf^0_*((f^{0})'\textup{sgn}(F^s)'\langle v'\rangle^k-f^0\textup{sgn}(F^s)\langle v\rangle^k)\mathrm{d}V+\int B^sf^0_*\textup{sgn}(F^s)'\langle v'\rangle^k(f^0-(f^0)')\mathrm{d}V\\
\colonequals &I_{41}+I_{42}.
\end{aligned}
\eeq
Similarly,
\beq\label{3.54}
\begin{aligned}
    &\langle Q^0(f^0,f^0),\textup{sgn}(F^s)\langle v\rangle^k\rangle=\int B^0 f^0_*f^0(\textup{sgn}(F^s)'\langle v'\rangle^k-\textup{sgn}(F^s)\langle v\rangle^k\mathrm{d}V\\
    =& \int B^0f^0_*((f^{0})'\textup{sgn}(F^s)'\langle v'\rangle^k-f^0\textup{sgn}(F^s)\langle v\rangle^k)\mathrm{d}V+\int B^0f^0_*\textup{sgn}(F^s)'\langle v'\rangle^k(f^0-(f^0)')\mathrm{d}V\\
    \colonequals &I_{43}+I_{44}.
\end{aligned}
\eeq
Therefore,
\ben\label{3.56}
I_4=\frac{I_{41}-I_{43}}{s}+\frac{I_{42}-I_{44}}{s}.
\een
Thanks to the cancellation lemma,
\beq\label{3.57}
\begin{aligned}
    I_{41}=\int b_s(\th)|v-v_*|^{\gamma}\Big(\frac{1}{\cos^{3+\gamma}\frac{\th}{2}}-1\Big)\textup{sgn}(F^s)f^0_*f^0\langle v\rangle^k\mathrm{d}V.
\end{aligned}
\eeq
\ben\label{3.58}
I_{43}=\int \frac{\mathbf{1}_{0\le\th\le \frac{\pi}{2}}}{2}|v-v_*|\Big(\frac{1}{\cos^{4}\frac{\th}{2}}-1\Big)\textup{sgn}(F^s)f^0_*f^0\langle v\rangle^k\mathrm{d}V.
\een
Note that
\beq\label{3.59}
\begin{aligned}
    b_s(\th)|v-v_*|^{\gamma}\Big(\frac{1}{\cos^{3+\gamma}\frac{\th}{2}}&-1\Big)-\frac{1}{2}|v-v_*|\Big(\frac{1}{\cos^{4}\frac{\th}{2}}-1\Big)
= (b_s(\th)-\frac{1}{2})|v-v_*|^{\gamma}\Big(\frac{1}{\cos^{3+\gamma}\frac{\th}{2}}-1\Big)\\
+&\frac{1}{2}(|v-v_*|^{\gamma}-|v-v_*|)\Big(\frac{1}{\cos^{3+\gamma}\frac{\th}{2}}-1\Big)
+ \frac{1}{2}|v-v_*|\Big(\frac{1}{\cos^{3+\gamma}\frac{\th}{2}}-\frac{1}{\cos^{4}\frac{\th}{2}}\Big).
\end{aligned}
\eeq
Let
\beq\label{3.60}
\begin{aligned}
    D_1=(b_s(\th)-\frac{1}{2})|v-v_*|^{\gamma}&\Big(\frac{1}{\cos^{3+\gamma}\frac{\th}{2}}-1\Big),\quad D_2=\frac{1}{2}(|v-v_*|^{\gamma}-|v-v_*|)\Big(\frac{1}{\cos^{3+\gamma}\frac{\th}{2}}-1\Big),\\
    &D_3=\frac{1}{2}|v-v_*|\Big(\frac{1}{\cos^{3+\gamma}\frac{\th}{2}}-\frac{1}{\cos^{4}\frac{\th}{2}}\Big).
\end{aligned}
\eeq
From \eqref{3.57}, \eqref{3.58} and \eqref{3.59},
\ben\label{3.61}
\frac{I_{41}-I_{43}}{s}=\int \frac{D_1+D_2+D_3}{s}\textup{sgn}(F^s)f^0_*f^0\langle v\rangle^k\mathrm{d}V.
\een
Note that by \eqref{3.29} and \eqref{3.60},
\ben\label{3.62}
|b_s(\th)-\frac{1}{2}|\le Cs\th^{-2-2s},\quad \Big(\frac{1}{\cos^{3+\gamma}\frac{\th}{2}}-1\Big)\le C\th^2\implies \frac{|D_1|}{s}\le C\th^{-2s}|v-v_*|^{\gamma}.
\een
By \eqref{3.62},
\ben\label{3.63}
\Big|\int \frac{D_1}{s}\textup{sgn}(F^s)f^0_*f^0\langle v \rangle^k\mathrm{d}V\Big|\le C|f^0|_{L^1_{k+\gamma}}|f^0|_{L^1_{\gamma}}\le C|f_{\textup{in}}|_{L^1_{k+\gamma}}|f_{\textup{in}}|_{L^1_{2}}.
\een
Thanks to the basic inequality \eqref{inequality},
$$\begin{cases}
    1-x^\al\le \al\ln x^{-1}\quad \al>0,x\in(0,1),\\
    x^\al-1\le \al x^\al\ln x \quad \al>0,x\in(1,\infty).
\end{cases}$$
Taking $\al=4s,x=|v-v_*|$ and recalling \eqref{1.3} $\gamma=1-4s$, one has for any $0<\d<1$,
\beq\label{3.64}
\begin{cases}
    |v-v_*|^{\gamma}-|v-v_*|=|v-v_*|^{\gamma}(1-|v-v_*|^{4s})\le |v-v_*|^{\gamma}4s\ln|v-v_*|^{-1}, \quad 0\le|v-v_*|<1,\\
    \Big||v-v_*|^{\gamma}-|v-v_*|\Big|\le |v-v_*|^{\gamma}(|v-v_*|^{4s}-1)\le Cs|v-v_*|\ln |v-v_*|\le C_{\delta}s|v-v_*|^{1+\delta}, \quad |v-v_*|>1.
\end{cases}
\eeq
Since $0<s<\frac{1}{8}\implies \gamma=1-4s>\frac{1}{2}$. We have for any $|v-v_*|\le1$,
$$|v-v_*|^{\gamma}4s\ln|v-v_*|^{-1}\le |v-v_*|^{\frac{1}{2}}4s\ln|v-v_*|^{-1}\le4s\max\limits_{0\le x\le 1}x^{\frac{1}{2}}\log x^{-1}\le Cs.$$ Here constant $C$ is independent of $s$.
By \eqref{3.64}, there exists an constant $C_\delta$,
\ben\label{3.65}
\Big||v-v_*|^{\gamma}-|v-v_*|\Big|\le C_{\delta}s\langle v-v_*\rangle^{1+\delta}.
\een
From \eqref{3.65} and \eqref{3.60}, $|\frac{D_2}{s}|\le C_{\delta}\langle v-v_*\rangle^{1+\delta}$. Therefore,
\beq\label{3.66}
\begin{aligned}
\Big|\int \frac{D_2}{s}\textup{sgn}(F^s)f^0_*f^0\langle v\rangle^k \mathrm{d}V\Big|\le  C_{\delta}\int_{\mathbb{R}^3\times\mathbb{R}^3}\langle v-v_*\rangle^{1+\delta}f^0_*f^0\langle v\rangle^k\mathrm{d}V\le C_{\delta}|f_{\textup{in}}|_{L^1_{1+k+\delta}}|f_{\textup{in}}|_{L^1_{2}}.
\end{aligned}
\eeq
By \eqref{inequality} and \eqref{3.60},
\ben\label{3.67}
|D_3|=\frac{1}{2}|v-v_*|\cos^{-4}\frac{\th}{2}(1-\cos^{4s}\frac{\th}{2})\le C|v-v_*|s\implies \frac{|D_3|}{s}\le C|v-v_*|.
\een
By \eqref{3.67},
\ben\label{3.68}
\Big|\int \frac{D_3}{s}\textup{sgn}(F^s)f^0_*f^0\langle v \rangle^k\mathrm{d}V\Big|\le C\int_{\mathbb{R}^3\times\mathbb{R}^3} |v-v_*|f^0_*f^0\langle v \rangle^k\mathrm{d}V\le C|f_{\textup{in}}|_{L^1_{k+1}}|f_{\textup{in}}|_{L^1_2}.
\een
Combining \eqref{3.61}, \eqref{3.63}, \eqref{3.66} and \eqref{3.68},
\ben\label{3.69}
\Big|\frac{I_{41}-I_{43}}{s}\Big|\le C_{\delta}|f_{\textup{in}}|_{L^1_{k+1+\delta}}|f_{\textup{in}}|_{L^1_{2}}.
\een
Recalling \eqref{3.52} and \eqref{3.54},
\ben\label{3.70}
\frac{I_{42}-I_{44}}{s}=\int \frac{B^s-B^0}{s}f^0_*\textup{sgn}(F^s)'\langle v'\rangle^k(f^0-(f^0)')\mathrm{d}V.
\een
Since
\ben\label{3.71}
B^s-B^0=(b_s(\th)-\frac{1}{2})|v-v_*|+b_s(\th)(|v-v_*|^{\gamma}-|v-v_*|).
\een
Define 
\ben\label{3.72}
E_1(|v-v_*|,\th)=(b_s(\th)-\frac{1}{2})|v-v_*|,\quad E_2(|v-v_*|,\th)=b_s(\th)(|v-v_*|^{\gamma}-|v-v_*|).
\een
From \eqref{3.70}, \eqref{3.71} and \eqref{3.72},
\ben\label{3.73}
\frac{I_{42}-I_{44}}{s}=\int \frac{E_1+E_2}{s}f^0_*\textup{sgn}(F^s)'\langle v'\rangle^k(f^0-(f^0)')\mathrm{d}V.
\een
Using the Taylor formula,
\beq\label{3.74}
\begin{aligned}
    &\Big|\int\frac{E_1}{s}f^0_*\textup{sgn}(F^s)'\langle v'\rangle^k(f^0-(f^0)')\mathrm{d}V\Big|= \Big|\int\frac{E_1}{s}f^0_*\textup{sgn}(F^s)'\langle v'\rangle^k\int_0^1\grad f^0(v(\kappa))\cdot(v'-v)\mathrm{d}\kappa\mathrm{d}V\Big|\\
    = & \Big|\int_0^1\int_{\mathbb{R}^3\times\mathbb{R}^3\times\mathbb{S}^2}\frac{E_1}{s}f^0_*\textup{sgn}(F^s)'\langle v'\rangle^k\grad f^0(v(\kappa))\cdot(v'-v)\mathrm{d}V\mathrm{d}\kappa \Big|\\
    \le &  \int_0^1\int_{\mathbb{R}^3\times\mathbb{R}^3\times\mathbb{S}^2}\frac{|E_1|}{s}f^0_*\langle v'\rangle^k|\grad f^0(v(\kappa))||v'-v|\mathrm{d}V\mathrm{d}\kappa .
\end{aligned}
\eeq
Note that
\ben\label{3.75}
|v'-v_*|=\cos\frac{\th}{2}|v-v_*|\le |v-v_*|,\quad |v'-v|=\sin\frac{\th}{2}|v-v_*|.
\een
From \eqref{3.75}, 
\beq\label{3.76}
    \langle v'\rangle\le \sqrt{2}(\langle v'-v_*\rangle+\langle v_*\rangle)\le \sqrt{2}(\langle v-v_*\rangle+\langle v_*\rangle).
\eeq
By \eqref{3.72}, \eqref{3.74}, \eqref{3.75} and \eqref{3.76},
\beq\label{3.77}
\begin{aligned}
    &\Big|\int\frac{E_1}{s}f^0_*\textup{sgn}(F^s)'\langle v'\rangle^k(f^0-(f^0)')\mathrm{d}V\Big|\\
    \le &\int_0^1\int_{\mathbb{R}^3\times\mathbb{R}^3\times\mathbb{S}^2}\frac{\Big|b_s(\th)-\frac{1}{2}\Big|\sin\frac{\th}{2}}{s}f^0_*\langle v'\rangle^k|\grad f^0(v(\kappa))||v-v_*|^2\mathrm{d}V\mathrm{d}\kappa
    \\
    \le & C_k\int_0^1\int_{\mathbb{R}^3\times\mathbb{R}^3\times\mathbb{S}^2}\frac{\Big|b_s(\th)-\frac{1}{2}\Big|\sin\frac{\th}{2}}{s}f^0_*|\grad f^0(v(\kappa))|(\langle v_*\rangle^k+\langle v-v_*\rangle^k)\langle v-v_*\rangle^2\mathrm{d}V\mathrm{d}\kappa\\
    \le &C_k\int_0^1\int_{\mathbb{R}^3\times\mathbb{R}^3\times\mathbb{S}^2}\frac{\Big|b_s(\th)-\frac{1}{2}\Big|\sin\frac{\th}{2}}{s}f^0_*|\grad f^0(v)|\Big(\langle v_*\rangle^k+\langle \frac{v-v_*}{\psi_\kappa(\th)}\rangle^k \Big)\langle \frac{v-v_*}{\psi_\kappa(\th)}\rangle^2\frac{1}{\psi_{\kappa}(\th)^3}\mathrm{d}V\mathrm{d}\kappa\\
    \le &C_k \int_{0}^{\frac{\pi}{2}}\frac{\Big|b_s(\th)-\frac{1}{2}\Big|\sin\frac{\th}{2}}{s}\mathrm{d}\sigma |f_{\textup{in}}|_{L^1_{k+2}}|\grad f^0|_{L^1_{k+2}}\\
    \eqref{3.29}\le &C_k|f_{\textup{in}}|_{L^1_{k+2}}|\grad f^0|_{L^1_{k+2}}.
\end{aligned}
\eeq
Here, we use the change of variable $v\to v(\kappa)$ from the third to the fourth line.

It remains to estimate $E_2$.
Recalling \eqref{3.72},
$$E_2(|v-v_*|,\th)=b_s(\th)(|v-v_*|^\gamma-|v-v_*|).$$
By \eqref{3.65},
\ben\label{3.78}
\Big|\frac{E_2}{s}\Big|\le C_\delta b_s(\th)\langle v-v_*\rangle^{1+\delta}.
\een
Therefore,
$$
\begin{aligned}
    &\Big|\int\frac{E_2}{s}f^0_*\textup{sgn}(F^s)'\langle v'\rangle^k(f^0-(f^0)')\mathrm{d}V\Big|\\
    \le &
    \int_0^1\int_{\mathbb{R}^3\times\mathbb{R}^3\times\mathbb{S}^2}\frac{|E_2|\sin\frac{\th}{2}}{s}f^0_*\langle v'\rangle^k|\grad f^0(v(\kappa))||v-v_*|\mathrm{d}V\mathrm{d}\kappa\\
    \eqref{3.78}\le & 
    C_\delta\int_0^1\int_{\mathbb{R}^3\times\mathbb{R}^3\times\mathbb{S}^2}b_s(\th)\sin\frac{\th}{2}f^0_*\langle v'\rangle^k|\grad f^0(v(\kappa))|\langle v-v_*\rangle^{2+\delta}\mathrm{d}V\mathrm{d}\kappa\\
    \eqref{3.76}\le &
    C_\delta\int_0^1\int_{\mathbb{R}^3\times\mathbb{R}^3\times\mathbb{S}^2}b_s(\th)\sin\frac{\th}{2}f^0_*(\langle v_*\rangle^k+\langle v-v_*\rangle^k)|\grad f^0(v(\kappa))|\langle v-v_*\rangle^{2+\delta}\mathrm{d}V\mathrm{d}\kappa.
\end{aligned}
$$
Taking change of variable $v\to v(\kappa)$, one has
\beq\label{3.80}
\begin{aligned}
    &\Big|\int\frac{E_2}{s}f^0_*\textup{sgn}(F^s)'\langle v'\rangle^k(f^0-(f^0)')\mathrm{d}V\Big|\\
     \le & C_\delta\int_0^1\int_{\mathbb{R}^3\times\mathbb{R}^3\times\mathbb{S}^2}b_s(\th)\sin\frac{\th}{2}f^0_*\Big(\langle v_*\rangle^k+\langle \frac{v-v_*}{\psi_\kappa(\th)}\rangle^k\Big)|\grad f^0(v)|\langle \frac{v-v_*}{\psi_\kappa(\th)}\rangle^{2+\delta}\frac{1}{\psi^3_\kappa(\th)}\mathrm{d}V\mathrm{d}\kappa\\
    \le &
    C(\delta,k)|f_{\textup{in}}|_{L^1_{k+2+\delta}}|\grad f^0|_{L^1_{k+2+\delta}}.
\end{aligned}
\eeq
Combining \eqref{3.73}, \eqref{3.77} and \eqref{3.80},
\ben\label{3.81}
\Big|\frac{I_{42}-I_{44}}{s}\Big|\le C(\delta,k)|f_{\textup{in}}|_{L^1_{k+2+\delta}}|\grad f^0|_{L^1_{k+2+\delta}}.
\een
From \eqref{3.56}, \eqref{3.69} and \eqref{3.81},
\ben\label{3.82}
|I_4|\le C(\delta,k)(|f_{\textup{in}}|_{L^1_{k+1+\delta}}|f_{\textup{in}}|_{L^1_{2}}+|f_{\textup{in}}|_{L^1_{k+2+\delta}}|\grad f^0|_{L^1_{k+2+\delta}}).
\een

Combining \eqref{3.34}, \eqref{3.48}, \eqref{3.51} and \eqref{3.82}, for $0\le t\le T$,
$$
\begin{aligned}
\frac{\mathrm{d}}{\mathrm{d}t}|F^s|_{L^1_k}\le &C_k|F^s|_{L^1_{k-1+\gamma}}(|f_{\textup{in}}|_{L^1_{k+\gamma}}+|\grad f^0|_{L^1_{k+1+\gamma}})\\
&+C(\delta,k)(|f_{\textup{in}}|_{L^1_2}|f_{\textup{in}}|_{L^1_{k+1+\delta}}+|f_{\textup{in}}|_{L^1_{k+2+\delta}}|\grad f^0|_{L^1_{k+2+\delta}})\\
\le &
C_k(|f_{\textup{in}}|_{L^1_{k+1}}+|\grad f^0|_{L^{\infty}([0,T];L^1_{k+2})})|F^s|_{L^1_k}\\
&+C(\delta,k)(|f_{\textup{in}}|_{L^1_2}|f_{\textup{in}}|_{L^1_{k+1+\delta}}+|f_{\textup{in}}|_{L^1_{k+2+\delta}}|\grad f^0|_{L^\infty([0,T];L^1_{k+2+\delta})}).
\end{aligned}
$$
Using Gronwall's inequality, for any $0<t<T$,
\beq\label{3.84}
\begin{aligned}
|F^s(t)|_{L^1_k}&\le \frac{C(\delta,k)(|f_{\textup{in}}|_{L^1_2}|f_{\textup{in}}|_{L^1_{k+1+\delta}}+|f_{\textup{in}}|_{L^1_{k+2+\delta}}|\grad f^0|_{L^\infty([0,T];L^1_{k+2+\delta})})}{C_k(|f_{\textup{in}}|_{L^1_{k+1}}+|\grad f^0|_{L^{\infty}([0,T];L^1_{k+2})})}\times\\
&\Big(\exp\{(|f_{\textup{in}}|_{L^1_{k+1}}+|\grad f^0|_{L^{\infty}([0,T];L^1_{k+2})})t\}-1\Big).
\end{aligned}
\eeq
From \eqref{3.84}, we have $|F^s(0)|_{L^1_k}=0$, which corresponds to $f^s(0,v)=f^0(0,v)=f_{\textup{in}}(v)$.
By Lemma \ref{lemma3.3}, there exists a constant $C$ depending on $k,\delta,|f_{\textup{in}}|_{L^1_{k+3+\delta}},|f_{\textup{in}}|_{W^{1,1}_{k+2+\delta}},|f_{\textup{in}}|_{L\log L},T$, such that \eqref{error-result} holds true.
\end{proof}

%\ben\label{3.85}
%\max\limits_{t\in[0,T]}|F^s(t)|_{L^1_{k}}\le C(k,\delta,|f_{\textup{in}}|_{L^1_{k+3+\delta}},|f_{\textup{in}}|_{W^{1,1}_{k+2+\delta}},|f_{\textup{in}}|_{L\log L},T).
%\een

\section*{Acknowledgments} 
The research was supported by the National Key Research and
Development Program of China under the grant  2023YFA1010300. 
Jin Woo Jang was supported by the National Research Foundation of Korea (NRF) under the grants RS-2023-00210484, RS-2023-00219980, and 2021R1A6A1A10042944. 
Yu-Long Zhou was also supported by NSF of China under the grants 12522112, 12471217,  Guangdong Basic and Applied Basic Research Foundation under the grant 2024A1515010495, Guangzhou Basic and Applied Basic Research Foundation under the grant 2025A04J7091. Zheng-An Yao was supported by the Research and Development Project of Pazhou Lab(Huangpu), Grant Number 2023K0601.

\bibliographystyle{plain}
\bibliography{references}
\end{document}